\documentclass{article}
\usepackage{amsmath}
\usepackage{amssymb}
\usepackage{amstext}
\usepackage{amsfonts}

\newtheorem{theorem}{Theorem}[section]

\newtheorem{definition}[theorem]{Definition}

\newtheorem{lemma}[theorem]{Lemma}

\newtheorem{proposition}[theorem]{Proposition}
\newtheorem{remark}[theorem]{Remark}

\newenvironment{proof}[1][Proof]{\noindent\textbf{#1.} }{\ \rule{0.5em}{0.5em}}
\numberwithin{equation}{section}

\def \loc {\mathrm{loc}}

\def \supp {\mathrm{supp}\,}

\def \loc {\mathrm{loc}}

\begin{document}

\title{Regularity for nonuniformly elliptic equations with $p,q-$growth and
explicit $x,u-$dependence}
\author{Giovanni Cupini$^{1}$, Paolo Marcellini$^{2}$\thanks{%
Corresponding author.}, Elvira Mascolo$^{2}$ \\
%EndAName
$\quad $ \\
{\normalsize $^{1}$Dipartimento di Matematica, Universit\`a di Bologna }\\
{\normalsize Piazza di Porta S. Donato 5, 40126 - Bologna, Italy }\\
{\normalsize giovanni.cupini@unibo.it }\\
{\normalsize $^{2}$Dipartimento di Matematica e Informatica ``U. Dini'',
Universit\`a di Firenze}\\
{\normalsize Viale Morgagni 67/A, 50134 - Firenze, Italy}\\
{\normalsize paolo.marcellini@unifi.it, elvira.mascolo@unifi.it}}
\date{}
\maketitle

\begin{abstract}
We are interested in the regularity of weak solutions $u$ to the elliptic
equation in divergence form as in (\ref{elliptic equation}); precisely in
their local boundedness and their local Lipschitz continuity under \textit{%
general growth conditions}, the so called $p,q-$\textit{growth conditions}
as in (\ref{ellipticity in the introduction}),(\ref{growth in the
introduction}) below. We found a unique set of assumptions to get all these
regularity properties at the same time; in the meantime we also found the
way to treat a more general context, with explicit dependence on $\left(
x,u\right) $, other than on the gradient variable $\xi =Du$; these aspects
require particular attention due to the $p,q-$context, with some differences
and new difficulties compared to the standard case $p=q$.
\end{abstract}

\tableofcontents

\bigskip

\emph{Key words}: Elliptic equations, Regularity of solutions, Local
Boundedness, Local Lipschitz continuity, Higher differentiability,
p,q-growth conditions, General growth conditions.

\emph{Mathematics Subject Classification (2020)}: Primary: 35D30, 35J15,
35J60; Secondary: 35B45, 49N60.

\section{Introduction}

We consider elliptic equations of the form 
\begin{equation}
\sum_{i=1}^{n}\frac{\partial }{\partial x_{i}}a^{i}\left( x,u\left( x\right)
,Du\left( x\right) \right) =b\left( x,u\left( x\right) ,Du\left( x\right)
\right) ,\;\;\;\;\;x\in \Omega \,,  \label{elliptic equation}
\end{equation}%
where the vector field $a\left( x,u,\xi \right) =\left( a^{i}\left( x,u,\xi
\right) \right) _{i=1,\ldots ,n}$ is locally Lipschitz continuous in $\Omega
\times \mathbb{R}\times \mathbb{R}^{n}$ and the right hand side $b\left(
x,u,\xi \right) $ is a Carath\'{e}odory function in $\Omega \times \mathbb{R}%
\times \mathbb{R}^{n}$. Our assumptions are characterized by the fact that
the second order elliptic equation in divergence form (\ref{elliptic
equation}) explicitly depends on $\left( x,u\right) \in \Omega \times 
\mathbb{R}$, where $\Omega $ is an open set in $\mathbb{R}^{n}$, $n\geq 2$,
other than on the gradient variable $\xi \in \mathbb{R}^{n}$. Further
characteristics are the ellipticity and growth condition in general form,
the so called $p,q-$\textit{conditions} as in (\ref{ellipticity in the
introduction}),(\ref{growth in the introduction}) below.

We are interested in the regularity of the weak solutions $u$ to (\ref%
{elliptic equation}); precisely in their local boundedness, their local
Lipschitz continuity and higher differentiability under \textit{general
growth conditions}. The effect of these results, the gradient $Du$ being 
\textit{a-posteriori} locally bounded, is that the growth properties of the
differential operator reduces to the so called \textit{natural growth
conditions }(a name which usually denotes the case $p=q$); thus, having in
force the local Lipschitz continuity considered in this manuscript, \textit{%
also the }$C^{1,\alpha }$\textit{\ regularity can be deduced by the
classical literature on regularity}.

A few words about the well-known classical results on regularity for weak
solutions to elliptic equations as in (\ref{elliptic equation}). A main tool
is the fundamental H\"{o}lder continuity result by De Giorgi \cite{De Giorgi
1957}, which has been extensively considered also in the book by
Ladyzhenskaya-Ural'tseva \cite[Chapter 4]{Ladyzhenskaya-Uraltseva 1968}. We
also refer to the article by Evans \cite{Evans 1982 JDE}, with explicit
dependence of $\left( a^{i}\right) _{i=1,\ldots ,n}$ in (\ref{elliptic
equation}) only on the gradient variable and with right hand side $b=0$; the
celebrated paper by DiBenedetto \cite{DiBenedetto 1983} on the $C^{1,\alpha
}-$regularity for weak solutions of a class of degenerate elliptic
equations; the famous $C^{1,\alpha }-$regularity result by Tolksdorf \cite%
{Tolksdorf 1984 JDE}; the article by Manfredi \cite{Manfredi 1988} on the $%
p- $Laplacian type integrals of the Calculus of Variations. Later, see also
the well known articles by Lieberman \cite{Lieberman 1991} and Marcellini 
\cite{Marcellini 1991}, the book by Giusti \cite{Giusti 2003 book}, the
results by Duzaar-Mingione \cite{Duzaar-Mingione 2010} and Cianchi-Maz'ya 
\cite{Cianchi-Mazya 2011}.

Our point of view are \textit{the }$p,q-$\textit{growth conditions} ($p\leq
q $) with respect to the gradient variable $\xi :=Du$, in this general
context with $\left( x,u\right) $ explicit dependence too. Precisely, the $%
p- $\textit{ellipticity} 
\begin{equation}
\sum_{i,j=1}^{n}\frac{\partial a^{i}}{\partial \xi _{j}}\lambda _{i}\lambda
_{j}\geq m\left( 1+\left\vert \xi \right\vert ^{2}\right) ^{\frac{p-2}{2}%
}\left\vert \lambda \right\vert ^{2}\;,
\label{ellipticity in the introduction}
\end{equation}%
valid for a positive constant $m$ and for every $\lambda ,\xi \in \mathbb{R}%
^{n}$ and $\left( x,u\right) \in \Omega \times \mathbb{R}$; and the $q-$%
\textit{growth }%
\begin{equation}
\left\vert \frac{\partial a^{i}}{\partial \xi _{j}}\right\vert \leq M\left(
\left( 1+\left\vert \xi \right\vert ^{2}\right) ^{\frac{q-2}{2}}+\left\vert
u\right\vert ^{\alpha }\right) \,,  \label{growth  in the introduction}
\end{equation}%
for constants $M>0$, $0\leq \alpha \leq \frac{2\left( q-2\right) }{q-p+2}$
(in particular $0\leq \alpha \leq q-2$ if $q=p$) and for all $\left( x,u,\xi
\right) \in \Omega \times \mathbb{R\times R}^{n}$ and $i,j\in \left\{
1,2,\ldots ,n\right\} $.

The interest in existence and regularity for weak solutions to elliptic
equations in divergence form, under general growth conditions, has risen in
the last decades, since the first related results in the 90's. Nowadays the
literature on $p,q-$problems is large and maybe there is not need to enter
too much into details. However we emphasize that a bound on the ratio $\frac{%
q}{p}$ of the type 
\begin{equation}
1\leq \tfrac{q}{p}<1+O\left( \tfrac{1}{n}\right)
\label{bound for the ratio q/p}
\end{equation}%
is necessary and, at the same time, another bound on the ratio $\frac{q}{p}$
of the same type is also sufficient for regularity. The first approach can
be found in \cite{Marcellini ARMA 1989},\cite{Marcellini 1991},\cite%
{Marcellini 1993},\cite{Marcellini JOTA 1996}. Not only in \textit{the }$%
p,q- $\textit{context} but also for \textit{general nonuniformly elliptic
problems}, see \cite{Marcellini 1993},\cite{Marcellini JOTA 1996},\cite%
{Marcellini AnnaliPisa 1996},\cite{DiMarco-Marcellini 2020}.

We already said that nowadays a large literature exists on this subject.
Recently Mingione gave a strong impulse with the introduction of the
terminology (and not only terminology, but also fine results obtained with
some colleagues of him, as detailed below) of \textit{double phase integrals}%
, of the type 
\begin{equation}
\int_{\Omega }\left\{ \tfrac{1}{p}\left\vert Du\right\vert ^{p}+\tfrac{1}{q}%
a\left( x\right) \left\vert Du\right\vert ^{q}\right\} \,dx\,
\label{double phase integral}
\end{equation}%
and also their nondegenerate version 
\begin{equation}
\int_{\Omega }\left\{ \tfrac{1}{p}\left( 1+\left\vert Du\right\vert
^{2}\right) ^{\frac{p}{2}}+\tfrac{1}{q}a\left( x\right) \left( 1+\left\vert
Du\right\vert ^{2}\right) ^{\frac{q}{2}}\right\} \,dx\,
\label{double phase integral (nondegenerate)}
\end{equation}%
whose Euler's first variation give rise to the differential equation in
divergence form 
\begin{equation}
\sum_{i=1}^{n}\frac{\partial }{\partial x_{i}}\left\{ \left( \left(
1+\left\vert Du\right\vert ^{2}\right) ^{\frac{p}{2}-1}+a\left( x\right)
\left( 1+\left\vert Du\right\vert ^{2}\right) ^{\frac{q}{2}-1}\right)
u_{x_{i}}\right\} =0,\;\;\;\;\;x\in \Omega \,.  \label{double phase equation}
\end{equation}%
Here the coefficient $a\left( x\right) $ is either H\"{o}lder continuous or
local Lipschitz continuous in $\Omega $. However, even more relevant, $%
a\left( x\right) $ is greater than or equal to zero in $\Omega $, with the
possibility to be equal to zero on a closed subset of $\Omega $. Therefore
the energy integral in (\ref{double phase integral}) or in (\ref{double
phase integral (nondegenerate)}), and the equation in (\ref{double phase
equation}) too, when $p<q$ behave like a $q-$Laplacian in the subset of $%
\Omega $ where $a\left( x\right) >0$ (and in this case the $p-$addendum
plays the role of a \textit{``lower order term"}), while it is a $p-$%
Laplacian in the subset of $\Omega $ where $a\left( x\right) =0$.

The \textit{double phase integrals in }(\ref{double phase integral}) and in (%
\ref{double phase integral (nondegenerate)}) are relevant examples \ of
energy integrals with $p,q-$growth; many other examples exist: $p-$power
times a logarithm, variable exponents $p\left( x\right) $, anisotropic
integrands such as $\sum_{i=1}^{n}\left\vert u_{x_{i}}\right\vert ^{p_{i}}$
with $p=\min_{i}\left\{ p_{i}\right\} $ and $q=\max_{i}\left\{ p_{i}\right\} 
$, and so on. Note in particular that the equation in (\ref{double phase
equation}) satisfies the $p,q-$growth conditions stated in (\ref{ellipticity
in the introduction}),(\ref{growth in the introduction}).

A special consideration about the recent interesting article by De
Filippis-Mingione \cite{DeFilippis-Mingione Inv 2023}, related to the
regularity of local minimizers of a class of integrals of the calculus of
variations which do not necessarily satisfy the Euler-Lagrange equation. In
order to allow  a comparison with the context considered here, since in \cite[Section  2.2]{DeFilippis-Mingione Inv 2023}  the energy integral is expressed in a {\em  splitting separated sum}  w.r.t. the gradient dependence and on the $u-$dependence, we discuss a  particular case of (\ref{elliptic equation}), 
\begin{equation}
	\sum_{i=1}^{n}\frac{\partial }{\partial x_{i}}\left\{ a^{i}\left(Du\right) \right\} =b\left( x,u\right) \,,
	\label{model example in the introduction}
\end{equation}%
where the vector field $\left\{ a^{i}\left( \xi \right) \right\}
_{i=1,2,\ldots ,n}$ in the left hand side does not explicitly depend  on $x$ and  $u$. In \cite[%
Section 2.2]{DeFilippis-Mingione Inv 2023} DeFilippis-Mingione allow a
very general $x-$dependence of the right hand side $b\left( x,u\right) $; in
fact they assume that the primitive $h\left( x,u\right) =\int_{0}^{u}b\left(
x,t\right) \,dt$ is a Carath\'{e}odory function, with $h\left( x,\cdot
\right) $ H\"{o}lder continuous and $h\left( \cdot ,u\right) $ in a suitable
Lorentz class for every fixed $u$. The opposite situation here: 
the vector field $\left\{ a^{i}\left(x,u,\xi \right) \right\}
_{i=1,2,\ldots ,n}$ depends on $(x,u)$ too, with  more strict
assumptions on the $x-$ and $u-$dependence 
 in
the right hand side $b\left( x,u,\xi \right) $, which here may also depend
on the gradient variable $\xi $. 
 Other differences are in force; for
instance in the context of this paper we allow general elliptic equations
without the symmetric assumption $\frac{\partial a^{i}}{\partial \xi _{j}}=%
\frac{\partial a^{j}}{\partial \xi _{i}}$, see (\ref{growth 1}) below.

Zhikov \cite{Zhikov 1987},\cite{Zhikov 1995} first studied similar kinds of
integrals with general growth in the context of \textit{homogenization} and
the \textit{Lavrentiev's phenomenon} (in this context see also the recent
article \cite{DeFilippis-Leonetti-Marcellini-Mascolo 2023}) More recent
contributions to regularity to minimizers of \textit{double phase integrals}
and for weak solutions to \textit{double phase equations} are due to Baroni,
Colombo and Mingione in \cite{Colombo-Mingione 2015},\cite%
{Baroni-Colombo-Mingione 2018}. \textit{Double phase} and \textit{variable
exponents} are studied by Byun-Oh \cite{Byun-Oh 2020}, Ragusa-Tachikawa \cite%
{Ragusa-Tachikawa 2020}, Fang-R\u{a}dulescu-Chao Zhang-Xia Zhang \textsc{%
\cite{Fang-Radulescu-Zhang-Zhang 2022},} Papageorgiou-R\u{a}dulescu-Zhang 
\cite{Papageorgiou-Radulescu-Zhang 2022},
Amoroso-Crespo-Blanco-Pucci-Winkert \cite%
{Amoroso-Crespo-Blanco-Pucci-Winkert 2023}, Arora-Fiscella-Mukherjee-Winkert 
\cite{Arora-Fiscella-Mukherjee-Winkert 2023}. In the context of \textit{%
variable exponents} and \textit{Orlicz--Sobolev spaces} we mention Mih\u{a}%
ilescu-Pucci-R\u{a}dulescu \cite{Mihailescu-Pucci-Radulescu 2008}, Zhang-R%
\u{a}dulescu \cite{Zhang-Radulescu 2018}, Ragusa-Razani-Safari \cite%
{Ragusa-Razani-Safari 2021}. For \textit{Orlicz--Sobolev spaces} see the
Springer Lecture Notes by Diening-Harjulehto-Hasto-Ruzicka \cite%
{Diening-Harjulehto-Hasto-Ruzicka 2011}, the reference paper by Chlebicka 
\cite{Chlebicka 2018}; see also Chlebicka-DeFilippis \cite%
{Chlebicka-DeFilippis 2019}, H\"{a}st\"{o}-Ok \cite{Hasto-Ok 2022}. About%
\textit{\ quasiconvex integrals} of the calculus of variations under general
growth conditions see in particular Cupini-Leonetti-Mascolo \cite%
{Cupini-Leonetti-Mascolo 2017} and \cite%
{Boegelein-Dacorogna-Duzaar-Marcellini-Scheven 2020},\cite%
{Dacorogna-Marcellini 1998},\cite{Marcellini 1984}; about partial
regularity, after Schmidt \cite{Schmidt 2009} more recently see DeFilippis 
\cite{De Filippis JMPA 2022} and DeFilippis-Stroffolini \cite{De
Filippis-Stroffolini 2023}.

Recently many authors obtained new regularity results, mainly in local
boundedness, higher summability, local Lipschitz continuity, $C^{1,\alpha }$%
. Most of the results deal with interior regularity, apart from
Cianchi-Maz'ya \cite{Cianchi-Mazya 2011},\cite{Cianchi-Mazya 2014}, B\"{o}%
gelein-Duzaar-Marcellini-Scheven \cite{Boegelein-Duzaar-Marcellini-Scheven
JMPA 2021}, DeFilippis-Piccinini \cite{Defilippis-Piccinini 2022-2023}, who
proved Lipschitz continuity of the weak solution up to the boundary. About
interior regularity for nonuniformly elliptic energy integrals with $p,q-$%
growth and general growth we quote for instance \cite{Bella-Schaffner 2020},%
\cite{Cupini-Marcellini-Mascolo 2017},\cite{Cupini-Marcellini-Mascolo 2018},%
\cite{Cupini-Marcellini-Mascolo-Passarelli 2023}; moreover Adimurthi-Tewary 
\cite{Adimurthi-Tewary 2022} and \cite{Eleuteri-Marcellini-Mascolo-Perrotta
2022},\cite{Eleuteri-Passarelli 2023},\cite{Zhang-Li 2023}. See also the
references in the review articles \cite{Marcellini 2020},\cite%
{Mingione-Radulescu 2021}.

What seems in the literature non often studied is the case when the
differential equation is not the Euler's first variation of an energy
integral. Not always the vector field $a\left( x,u,\xi \right) =\left(
a^{i}\left( x,u,\xi \right) \right) _{i=1,\ldots ,n}$ in the left hand side
of the equation (\ref{elliptic equation}) is the gradient, with respect to
the variable $\xi \in \mathbb{R}^{n}$, of a function $f\left( x,u,\xi
\right) $; on the contrary in the literature on this subject often the
condition $a^{i}=\partial f/\partial \xi _{i}$ , which simplifies the
framework, is one of the main assumption. If $f$ is of class $C^{2}$ in $\xi 
$ then this variational assumption $a^{i}=\partial f/\partial \xi _{i}$
implies that 
\begin{equation*}
\frac{\partial a^{i}}{\partial \xi _{j}}=\frac{\partial ^{2}f}{\partial \xi
_{i}\partial \xi _{j}}=\frac{\partial ^{2}f}{\partial \xi _{j}\partial \xi
_{i}}=\frac{\partial a^{j}}{\partial \xi _{i}},
\end{equation*}%
so that the $n\times n$ matrix $\left( \frac{\partial a^{i}}{\partial \xi
_{j}}\right) $ is symmetric. In this manuscript \textit{we do not assume
that this matrix is symmetric}. Instead we assume a condition of the type
(when $u$ is bounded; see more precisely below in (\ref{growth 1}))%
\begin{equation}
\left\vert \frac{\partial a^{i}}{\partial \xi _{j}}-\frac{\partial a^{j}}{%
\partial \xi _{i}}\right\vert \leq M\left( 1+\left\vert \xi \right\vert
^{2}\right) ^{\frac{p+q-4}{4}},  \label{not symmetry}
\end{equation}%
which requires an \textit{``intermediate growth"} with respect to $%
\left\vert \xi \right\vert $ (in fact the exponent $\frac{p+q-4}{2}$ is 
\textit{``intermediate"}, i.e. it is the average between $p-2$ and $q-2$) of
the antisymmetric terms of the matrix $\left( \frac{\partial a^{i}}{\partial
\xi _{j}}\right) $. Condition (\ref{not symmetry}) is automatically
satisfied, for instance, either when the vector field $\left( a^{i}\left(
x,u,\xi \right) \right) _{i=1,\ldots ,n}$ has the variational structure $%
a^{i}=\partial f/\partial \xi _{i}$, or if $p=q$. This second possibility
explains why in the literature the assumption (\ref{not symmetry}) is not
considered under the so called \textit{natural growth conditions }(i.e.,
when $p=q$); in fact if $p=q$ then (\ref{not symmetry}) is an elementary
consequence of the triangular inequality and of the growth assumption (\ref%
{growth in the introduction}) when $u$ is bounded.

An example can be easily constructed by adding in the left hand side of the
differential equation (\ref{double phase equation}) \textit{a perturbation
of the }$r-$\textit{Laplacian}; i.e. an operator of the type 
\begin{equation*}
\sum_{i=1}^{n}\frac{\partial }{\partial x_{i}}\left\{ a_{0}\left( x\right)
\left( \left( \left( 1+\left\vert Du\right\vert ^{2}\right) ^{\frac{r}{2}%
-1}\right) u_{x_{i}}+g^{i}\left( Du\right) \right) \right\} \,,
\end{equation*}%
with not symmetric and sufficiently small matrix $\left( \frac{\partial g^{i}%
}{\partial \xi _{j}}\right) $, in order to maintain the ellipticity of this
operator. If $p\leq r\leq \frac{p+q}{2}$ then the vector field, obtained by 
\textit{the sum of the two operators}, is $p-$elliptic, has $q-$growth,
satisfies condition (\ref{not symmetry}), but has not a variational
structure, in the sense that it is not the Euler's first variation of an
energy integral.

The right hand side $b$ in (\ref{elliptic equation}) may explicitly depend
on all variables $\left( x,u,\xi \right) \in \Omega \times \mathbb{R}\times 
\mathbb{R}^{n}$. Our growth assumption for $b$ is 
\begin{equation}
\left\vert b\left( x,u,\xi \right) \right\vert \leq M\left( 1+\left\vert \xi
\right\vert ^{2}\right) ^{\frac{p+q-2}{4}}+M\left\vert u\right\vert ^{\delta
-1}+b_{0}\left( x\right) \,  \label{growth of b  in the introduction}
\end{equation}%
a.e.$\;x\in \Omega $, $1\leq \delta <\frac{np}{n-p}:=p^{\ast }$ (when $n>p$)
and $b_{0}\in L_{\mathrm{loc}}^{s_{0}}\left( \Omega \right) $ for some $%
s_{0}>n$. The general growth condition in the right hand side of (\ref%
{growth of b in the introduction}) explicitly depends on $\left( x,u,\xi
\right) $ and does not allow us to assume the summability considered for
instance by DeFilippis-Mingione \cite{DeFilippis-Mingione ARMA 2021} and by
DeFilippis-Piccinini \cite{Defilippis-Piccinini 2022-2023}, with $b_{0}$ in
the Lorentz space $L\left( n,1\right) \left( \Omega \right) $, but however
with $M=0$ in (\ref{growth of b in the introduction}), and with the
functional inclusion $L^{n+\varepsilon }\left( \Omega \right) \subset
L\left( n,1\right) \left( \Omega \right) \subset L^{n}\left( \Omega \right) $
for all $\varepsilon >0$. Note however that the \textit{Lebesgue summability 
}$b_{0}\in L_{\mathrm{loc}}^{s_{0}}\left( \Omega \right) $\textit{\ with }$%
s_{0}>n$ is naturally assumed in the literature for regularity, see for
instance Colombo-Figalli \cite[Theorem 1.1]{Colombo-Figalli 2014} and \cite[%
Theorem 1.1]{Boegelein-Duzaar-Marcellini-Scheven JMPA 2021}. Moreover it is
sharp: in fact the weaker condition $b_{0}\in L_{\mathrm{loc}}^{n}\left(
\Omega \right) $ is not sufficient for the local Lipschitz continuity of the
weak solution.

The principal part of the equation also depends on all variables $\left(
x,u,\xi \right) \in \Omega \times \mathbb{R}\times \mathbb{R}^{n}$ and the
growth assumptions for the variables $\left( x,u\right) $ are (see the
details in (\ref{growth 0-b}), (\ref{growth 2}))%
\begin{equation}
\left\vert \frac{\partial a^{i}}{\partial u}\right\vert \leq M\left( \left(
1+\left\vert \xi \right\vert ^{2}\right) ^{\frac{p+q-4}{4}}+\left\vert
u\right\vert ^{\beta -1}\right) \,,\;\;\;\;\;\;\left\vert \frac{\partial
a^{i}}{\partial x_{s}}\right\vert \leq M\left( \left\vert u\right\vert
\right) \left( 1+\left\vert \xi \right\vert ^{2}\right) ^{\frac{p+q-2}{4}},
\label{growth wrt (x,u)  in the introduction}
\end{equation}%
for $x\in \Omega $,$\ 0\leq \beta <\frac{n\left( p-1\right) }{n-p}:=\frac{p-1%
}{p}p^{\ast }$ and for every $\xi \in \mathbb{R}^{n}$, $i,s=1,2,\ldots ,n$.
With respect to the $x-$dependence we could expect a more general assumption
depending on a Sobolev summability; for instance of the type 
\begin{equation*}
\left\vert \frac{\partial a^{i}}{\partial x_{s}}\right\vert \leq M\left(
x,\left\vert u\right\vert \right) \left( 1+\left\vert \xi \right\vert
^{2}\right) ^{\frac{p+q-2}{4}},
\end{equation*}%
where, for $\left\vert u\right\vert \leq L$ with $L$ fixed, $M\in L_{\mathrm{%
loc}}^{r}\left( \Omega \right) $ for some $r>n$. In this case we believe
that the local Lipschitz continuity of the weak solution should be proved
with a bound on the ration $q/p$ as in (\ref{bound for the ratio q/p})
depending on $r$ too; precisely $\frac{q}{p}<1+\frac{1}{n}-\frac{1}{r}$.
This Lebesgue summability was introduced in a simpler context by
Eleuteri-Marcellini-Mascolo \cite{Eleuteri-Marcellini-Mascolo 2019} and
later considered by DeFilippis-Mingione \cite{DeFilippis-Mingione 2020} too.
Note that (\ref{growth wrt (x,u) in the introduction})$_{2}$ corresponds to $%
r=+\infty $.

This research takes its origin from the two author's papers \cite%
{Cupini-Marcellini-Mascolo 2023},\cite{Marcellini 2023}, where the local
boundedness and the local Lipschitz continuity of weak solutions of the
equation (\ref{elliptic equation}) was studied, although with different
assumptions in the two cases. The effort here was to find a unique set of
assumptions to get all these regularity properties at the same time; in the
meantime we also found the way to treat a more general context. The final
regularity results, the local Lipschitz continuity and the higher
differentiability, are stated in Section \ref{Section: Main results}; the
starting step for these regularity results, i.e. the local boundedness of
the weak solutions, is stated and proved in Section \ref{Section: Local
boundedness}.

In our opinion a relevant application of the regularity results proved in
this manuscript relies in the authors' forthcoming paper \cite%
{Cupini-Marcellini-Mascolo Leray-Lions 2024}, devoted to the \textit{%
existence} $-$ in the sense of Leray-Lions \cite{Leray-Lions 1965} $-$ of
weak solutions $W^{1,q}\left( \Omega \right) $ to the Dirichlet problem
associated to the elliptic differential equation (\ref{elliptic equation}):
an aspect which requires particular attention due to the $p,q-$context when $%
q\neq p$, with some differences and new difficulties compared to the
standard case $p=q$, and with a crucial application of the local regularity
estimates obtained here.

\section{Main results\label{Section: Main results}}

\subsection{Assumptions\label{Section: assumptions for the local Lipschitz
continuity}}

We study the elliptic equations (\ref{elliptic equation}) under the
following \textit{general growth conditions} on the gradient variable $\xi
=Du$, named $p,q-$\textit{conditions}. In order to state the assumptions on
vector field $a=a\left( x,u,\xi \right) =\left( a^{i}\left( x,u,\xi \right)
\right) _{i=1,\ldots ,n}$, we start by the \textit{ellipticity}, valid for
some exponents $p,q$ ($2\leq p\leq q<p+1$), a constant $m>0$ and for every $%
\lambda ,\xi \in \mathbb{R}^{n}$, $\left( x,u\right) \in \Omega \times 
\mathbb{R}$, 
\begin{equation}
\sum_{i,j=1}^{n}\frac{\partial a^{i}}{\partial \xi _{j}}\lambda _{i}\lambda
_{j}\geq m(1+\left\vert \xi \right\vert ^{2})^{\frac{p-2}{2}}\left\vert
\lambda \right\vert ^{2}\;.  \label{ellipticity}
\end{equation}%
For the same $p,q$, $\left( x,u,\xi \right) \in \Omega \times \mathbb{%
R\times R}^{n}$ and for all $i,j=1,2,\ldots ,n$ we consider the \textit{%
growth conditions} for a constant $M>0$,%
\begin{equation}
\left\vert \frac{\partial a^{i}}{\partial \xi _{j}}\right\vert \leq
M(1+\left\vert \xi \right\vert ^{2})^{\frac{q-2}{2}}+M\left\vert
u\right\vert ^{\alpha }\,,  \label{growth 0-a}
\end{equation}%
\textit{\ }%
\begin{equation}
\left\vert \frac{\partial a^{i}}{\partial u}\right\vert \leq M(1+\left\vert
\xi \right\vert ^{2})^{\frac{p+q-4}{4}}+M\left\vert u\right\vert ^{\beta
-1}\,,  \label{growth 0-b}
\end{equation}
\begin{equation}
0\leq \alpha \leq \frac{2\left( q-2\right) }{q-p+2}\quad \text{and}\quad
0\leq \beta <\frac{n\left( p-1\right) }{n-p}:=\frac{p-1}{p}p^{\ast }
\label{growth 0-bGC}
\end{equation}%
(respectively, as in Lemma \ref{Lemma 1} and in (\ref{e:stimaIIesponente})
below; the upper bound on $\alpha $ is a perturbation of the inequality $%
\alpha \leq q-2$, valid in the particular case $q=p$).

Moreover, for every open set $\Omega ^{\prime }$, whose closure is contained
in $\Omega ,$ and for every $L>0$, there exists a positive constant $M\left(
L\right) $ (depending on $\Omega ^{\prime }$ and $L$) such that, for every $%
x\in \Omega ^{\prime }$, $\xi \in \mathbb{R}^{n}$ and for$\;\left\vert
u\right\vert \leq L$,\textit{\ }%
\begin{equation}
\left\vert \frac{\partial a^{i}}{\partial \xi _{j}}-\frac{\partial a^{j}}{%
\partial \xi _{i}}\right\vert \leq M\left( L\right) (1+\left\vert \xi
\right\vert ^{2})^{\frac{p+q-4}{4}},  \label{growth 1}
\end{equation}%
\begin{equation}
\left\vert \frac{\partial a^{i}}{\partial x_{s}}\right\vert \leq M\left(
L\right) (1+\left\vert \xi \right\vert ^{2})^{\frac{p+q-2}{4}},
\label{growth 2}
\end{equation}%
$i,j,s=1,2,\ldots ,n$. We also assume 
\begin{equation}
\left\vert a\left( x,0,0\right) \right\vert \in L_{\mathrm{loc}}^{\gamma
}\,,\;\;\;\;\;\forall \;i=1,2,\ldots ,n  \label{growth 3}
\end{equation}%
for an exponent $\gamma >\frac{n}{p-1}$; this summability condition is due
to (\ref{(i) in the Section}),(\ref{(ii) in the Section}) below, with $%
s_{1}:=\gamma \frac{p-1}{p}>\frac{n}{p}$. In particular (\ref{growth 3}) is
satisfied if $a\left( x,0,0\right) $ is a constant vector, or more generally
if it is locally bounded in $\Omega $.

The Carath\'{e}odory function $b\left( x,u,\xi \right) $ in the right hand
side of (\ref{elliptic equation}), $b:\Omega \times \mathbb{R}\times \mathbb{%
R}^{n}\rightarrow \mathbb{R}$, for a nonnegative constant $M$ satisfies the
condition 
\begin{equation}
\left\vert b\left( x,u,\xi \right) \right\vert \leq M(1+\left\vert \xi
\right\vert ^{2})^{\frac{p+q-2}{4}}+M\left\vert u\right\vert ^{\delta
-1}+b_{0}\left( x\right) \,  \label{b locally bounded}
\end{equation}%
a.e.$\;x\in \Omega $, for the same exponents $p,q$, for $1\leq \delta <\frac{%
np}{n-p}:=p^{\ast }$ (see (\ref{e:stimaIIesponente})) and for $b_{0}\in L_{%
\mathrm{loc}}^{s_{0}}\left( \Omega \right) $ with $s_{0}>n$ (see (\ref%
{Section 6 - formula 11}),(\ref{second derivatives})).

\subsection{Statement of a first regularity result}

In the context of general growth conditions it is necessary to specify the
functional class where to look for weak solutions. For instance, if we are
studying a double phase equation as in (\ref{double phase equation}), then
this class can be defined by the functions in $W_{\mathrm{loc}}^{1,1}\left(
\Omega \right) $ which make finite the energy integral (\ref{double phase
integral (nondegenerate)}), or in the degenerate case the integral in (\ref%
{double phase integral}) which turns out to be an equivalent condition. More
generally, if the differential equation is the Euler's first variation of an
energy integral, then the natural class of functions were to look for a
minimizer is the class which makes finite the energy integral. More
difficult is the general case of a differential equation which is not the
Euler's first variation of an integral, like in the context considered in
this manuscript.

Under $p,q-$growth conditions with $p\leq q$ the Sobolev class $W_{\mathrm{%
loc}}^{1,q}\left( \Omega \right) $ is the natural class were to look for
solutions; see Definition \ref{aweaksolDEF} below, see also \cite[Section 3.1%
]{Marcellini 2023} for a discussion about this aspect. At this stage we also
use the summability $b_{0}\in L_{\mathrm{loc}}^{q\prime }\left( \Omega
\right) $, which is consequence of the assumption $b_{0}\in L_{\mathrm{loc}%
}^{s_{0}}\left( \Omega \right) $ with $s_{0}>n$; in fact, since $q\geq 2$, $%
s_{0}>n\geq 2\geq \frac{q}{q-1}:=q^{\prime }$. We note that in some cases,
by using the a-priori regularity result in $W_{\mathrm{loc}}^{1,\infty
}\left( \Omega \right) $, we can also show the existence of weak solutions
of the associated Dirichlet problems in the Sobolev class $W^{1,p}\left(
\Omega \right) \cap W_{\mathrm{loc}}^{1,q}\left( \Omega \right) $; see \cite[%
Section 4]{Marcellini 1991},\cite{Cupini-Marcellini-Mascolo 2014},\cite%
{Eleuteri-Marcellini-Mascolo 2019},\cite{Cupini-Marcellini-Mascolo
Leray-Lions 2024}. When the differential equation is the Euler's first
variation of an energy integral then we can look directly to minimizers. In
this case we refer for instance to the higher integrability results for
minimizers of energy integrals in \cite[Theorem 2.1 and Remark 2.1]%
{Marcellini JOTA 1996} and in the well known article by
Esposito-Leonetti-Mingione \cite{Esposito-Leonetti-Mingione 1999}.

A main step in the proof of the following Theorem \ref{final theorem} is the
local boundedness result for the weak solution obtained in Section \ref%
{Section: Local boundedness}. Therefore we adopt here the bound in (\ref%
{e:hpq}) 
\begin{equation}
\tfrac{q}{p}<1+\tfrac{1}{n}\,.  \label{assumption on p,q}
\end{equation}%
However, if the weak solution $u$\ to the PDE (\ref{elliptic equation}) is
a-priori also locally bounded, i.e. if a-priori $u\in W_{\mathrm{loc}%
}^{1,q}\left( \Omega \right) \cap L_{\mathrm{loc}}^{\infty }\left( \Omega
\right) $, then Theorem \ref{final theorem} also holds under the bound $%
\frac{q}{p}<1+\frac{2}{n}$ instead of (\ref{assumption on p,q}); see \cite[%
Theorem 3.3]{Marcellini 2023}.

\begin{theorem}
\label{final theorem}Under the ellipticity and growth conditions (\ref%
{ellipticity})-(\ref{b locally bounded}), if the exponents $p,q$, with  $2\le p\le q<p+1$, 
satisfy the bound $\frac{q}{p}<1+\frac{1}{n}$, then every weak solution $%
u\in W_{\mathrm{loc}}^{1,q}\left( \Omega \right) $\ to the differential
equation (\ref{elliptic equation}) is of class $W_{\mathrm{loc}}^{1,\infty
}\left( \Omega \right) \cap W_{\mathrm{loc}}^{2,2}\left( \Omega \right) $.
Moreover, for every open set $\Omega ^{\prime }$ compactly contained in $%
\Omega $, there exist constants $c,c^{\prime },\alpha ,\gamma >0$\
(depending on the $L^{\infty }\left( \Omega ^{\prime }\right) $ norm of $u$
and on the data, but not on $u$) such that, for every $\varrho $\ and $R$\
with $0<\rho <R$ and $B_{R}(x_{0})\subset \Omega ^{\prime }$, 
\begin{equation}
\left\Vert Du\left( x\right) \right\Vert _{L^{\infty }\left( B_{\varrho };%
\mathbb{R}^{n}\right) }\leq \left( \frac{c}{\left( R-\varrho \right) ^{n}}%
\int_{B_{R}}(1+\left\vert Du\left( x\right) \right\vert ^{2})^{\frac{p}{2}%
}\,dx\right) ^{\frac{\alpha }{p}}  \label{interpolation bound 2}
\end{equation}%
\begin{equation*}
\underset{\text{for }n>2}{=}\;\left( \frac{c}{\left( R-\varrho \right) ^{n}}%
\left\Vert (1+\left\vert Du\left( x\right) \right\vert ^{2})^{\frac{1}{2}%
}\right\Vert _{L^{p}\left( B_{R}\right) }^{p}\right) ^{\frac{2}{\left(
n+2\right) p-nq}}\,;
\end{equation*}%
for the $n\times n$ matrix $D^{2}u$ of the second derivatives of $u$ the
following estimates hold 
\begin{equation}
\int_{B_{\rho }}\left\vert D^{2}u\right\vert ^{2}\,dx\leq \frac{c}{\left(
R-\rho \right) ^{2}}\int_{B_{R}}(1+\left\vert Du\left( x\right) \right\vert
^{2})^{\frac{q}{2}}\,dx\,  \label{bound on second derivatives}
\end{equation}%
\begin{equation*}
\leq \frac{c^{\prime }}{\left( R-\rho \right) ^{2}}\left( \frac{1}{\left(
R-\varrho \right) ^{\gamma \vartheta \left( q-p\right) }}\int_{B_{R}}(1+%
\left\vert Du\left( x\right) \right\vert ^{2})^{\frac{p}{2}}\,dx\right) ^{%
\frac{\alpha q}{\vartheta p}}\,.
\end{equation*}
\end{theorem}

The explicit expression of $\alpha $ in the gradient bound (\ref%
{interpolation bound 2}) is given by 
\begin{equation}
\alpha :=\frac{\vartheta \frac{p}{q}}{1-\vartheta \left( 1-\frac{p}{q}%
\right) }\;\underset{\text{for }n>2}{=}\;\frac{2p}{\left( n+2\right) p-nq}\,,
\label{alpha}
\end{equation}%
where $\vartheta \geq 1$ is the value $\vartheta :=\frac{2^{\ast }-2}{%
2^{\ast }\frac{p}{q}-2}\;\;\underset{\text{for }n>2}{=}\;\;\frac{2q}{%
np-\left( n-2\right) q}$ and $\gamma =\frac{n}{q}\vartheta $. Note that $%
\alpha \geq 1$ (in fact $\vartheta \frac{p}{q}\geq 1-\vartheta \left( 1-%
\frac{p}{q}\right) =1-\vartheta +\vartheta \frac{p}{q}$ is equivalent to $%
\vartheta \geq 1$) and, for the same reason, $\alpha =1$ if and only if $%
\vartheta =1$, which is equivalent to $q=p$.

\subsection{Statement of a second regularity result}

We may look for a $W_{\mathrm{loc}}^{1,\infty }\left( \Omega \right) -$%
regularity result when in the above growth assumptions the average exponent $%
\frac{p+q}{2}$, middle point between $p$ \ and $q$, is replaced by $q$; in
this case - a priori - we consider less restrictive assumptions. More
precisely, in view of the next Theorem \ref{final theorem 2}, compared with
Section \ref{Section: assumptions for the local Lipschitz continuity} we now
have the less restrictive growth conditions (obtained when in (\ref{growth
0-b}),(\ref{growth 2}),(\ref{b locally bounded}) $\frac{p+q}{2}$ is replaced
by $q$) 
\begin{equation}
\left\{ 
\begin{array}{l}
\left\vert \frac{\partial a^{i}}{\partial \xi _{j}}\right\vert \leq
M(1+\left\vert \xi \right\vert ^{2})^{\frac{q-2}{2}}+M\left\vert
u\right\vert ^{\alpha } \\ 
\left\vert \frac{\partial a^{i}}{\partial u}\right\vert \leq M(1+\left\vert
\xi \right\vert ^{2})^{\frac{q-2}{2}}+M\left\vert u\right\vert ^{\beta -1}
\\ 
\left\vert \frac{\partial a^{i}}{\partial x_{s}}\right\vert _{\;}\leq
M\left( L\right) (1+\left\vert \xi \right\vert ^{2})^{\frac{q-1}{2}} \\ 
\left\vert a^{i}\left( x,0,0\right) \right\vert \in L_{\mathrm{loc}}^{\gamma
^{\;}} \\ 
\left\vert b\left( x,u,\xi \right) \right\vert \leq M(1+\left\vert \xi
\right\vert ^{2})^{\frac{q-1}{2}}+M\left\vert u\right\vert ^{\delta
-1}+b_{0}\left( x\right) \,%
\end{array}%
\right.  \label{growth assumptions 2}
\end{equation}%
for every $\xi \in \mathbb{R}^{n}$, $i,j,s=1,2,\ldots ,n$, $u$ as before.
The parameters satisfy the bounds 
\begin{equation}
\left\{ 
\begin{array}{l}
0\leq \alpha \leq \frac{2q-p-2}{q-p+1} \\ 
0\leq \beta <\frac{n\left( p-1\right) }{n-p}=\frac{p-1}{p}p^{\ast } \\ 
\gamma >\frac{n}{p-1} \\ 
1\leq \delta <\frac{np}{n-p}=p^{\ast } \\ 
b_{0}\in L_{\mathrm{loc}}^{s_{0}}\left( \Omega \right) ,\;\;s_{0}>n\,.%
\end{array}%
\right. .  \label{growth assumptions 3}
\end{equation}%
When in (\ref{growth 1}) we replace $\frac{p+q}{2}$ by $q$ we get $%
\left\vert \frac{\partial a^{i}}{\partial \xi _{j}}-\frac{\partial a^{j}}{%
\partial \xi _{i}}\right\vert \leq M\left( L\right) (1+\left\vert \xi
\right\vert ^{2})^{\frac{q-2}{4}}$; we do not need to state it among the
other growth assumptions, since this time it is automatically satisfied as
consequence of (\ref{growth assumptions 2})$_{1}$ when $\left\vert
u\right\vert \leq L$.

As we said above, the following Theorem \ref{final theorem 2} holds with the
growth assumptions in (\ref{growth assumptions 2}), which are less
restrictive than those of Theorem \ref{final theorem}, however with a
stronger a-priori summability condition on the weak solution $u$, a more
restrictive bound for $\alpha $ in (\ref{growth assumptions 3})$_{1}$ and
for the quotient $\frac{q}{p}$.

\begin{theorem}
\label{final theorem 2}\textit{Under the ellipticity (\ref{ellipticity}) and
the }$p,q-$\textit{growth conditions (\ref{growth assumptions 2}),(\ref%
{growth assumptions 3}), if}   the exponents $p,q$, with  $2\le p\le q<p+2$, 
satisfy the bound $\frac{q}{p}<1+%
\frac{1}{2n}\,$, then every weak solution $u\in W_{\mathrm{loc}%
}^{1,2q-p}\left( \Omega \right) $ to the elliptic equation (\ref{elliptic
equation}) is of class $W_{\mathrm{loc}}^{1,\infty }\left( \Omega \right)
\cap W_{\mathrm{loc}}^{2,2}\left( \Omega \right) $ and, for every open set $%
\Omega ^{\prime }$ whose closure is contained in $\Omega ,$ there exist
constants $c,\alpha _{1},R_{0}>0$\ (depending on the $L^{\infty }\left(
\Omega ^{\prime }\right) $ norm of $u$ and on the data, but not on $u$) such
that, for every $\varrho $\ and $R$\ with $0<\rho <R$ and $%
B_{R}(x_{0})\subset \Omega ^{\prime }$, 
\begin{equation}
\left\Vert Du\left( x\right) \right\Vert _{L^{\infty }\left( B_{\varrho };%
\mathbb{R}^{n}\right) }\leq \left( \frac{c}{\left( R-\varrho \right) ^{n}}%
\int_{B_{R}}(1+\left\vert Du\left( x\right) \right\vert ^{2})^{\frac{p}{2}%
}\,dx\right) ^{\frac{\alpha _{1}}{p}}  \label{interpolation bound 3}
\end{equation}%
\begin{equation*}
\underset{\text{for }n>2}{=}\;\left( \frac{c}{\left( R-\varrho \right) ^{n}}%
\left\Vert (1+\left\vert Du\left( x\right) \right\vert ^{2})^{\frac{1}{2}%
}\right\Vert _{L^{p}\left( B_{R}\right) }^{p}\right) ^{\frac{1}{\left(
n+1\right) p-nq}}\,.
\end{equation*}%
Moreover the $L^{2}-$local estimates (\ref{bound on second derivatives}) for
the matrix $D^{2}u$ of the second derivatives hold when we replace $q$ by $%
2q-p$.
\end{theorem}

The exponent $\alpha _{1}$ in (\ref{interpolation bound 3}), when $n>2$, is
given by 
\begin{equation}
\alpha _{1}:=\;\frac{p}{\left( n+1\right) p-nq}\,.  \label{alpha'}
\end{equation}%
Similarly to $\alpha $ in (\ref{alpha}), also $\alpha _{1}$ is equal to $1$
if and only if $q=p$; note that $\alpha _{1}$ is well defined as a positive
real number since $\left( n+1\right) p-nq>0$, being this equivalent to the
bound: $\frac{q}{p}<1+\frac{1}{n}$. We recall again that Theorem \ref{final
theorem 2} holds for general differential equations (\ref{elliptic equation}%
) without the symmetric assumption $\frac{\partial a^{i}}{\partial \xi _{j}}=%
\frac{\partial a^{j}}{\partial \xi _{i}}$.

\begin{remark}
If a-priori we know that the weak solution is locally bounded, then Theorem %
\ref{final theorem} also holds under the weaker bound $\frac{q}{p}<1+\frac{2%
}{n}$, while Theorem \ref{final theorem 2} also holds under the bound $\frac{%
q}{p}<1+\frac{1}{n}$. The reason relies on the fact that in this case it is
not necessary to apply the local boundedness result of Theorem \ref%
{t:mainboundeness} and we can modify and adapt the method in \cite%
{Marcellini 2023} to the context considered here.
\end{remark}

\subsection{The classical case under the so-called natural growth conditions}

Finally we observe that all the results of this paper hold in the particular
case $q=p$, when the bound $\frac{q}{p}<1+\frac{1}{n}$ in (\ref{assumption
on p,q}) of course is satisfied. When $q=p$ the statements of the two
Theorems \ref{final theorem} and \ref{final theorem 2} coincide each other;
the ellipticity (\ref{ellipticity}) and the growth conditions (\ref{growth
assumptions 2}),(\ref{growth assumptions 3}) can be read with $q=p$ for the
validity of the regularity in $W_{\mathrm{loc}}^{1,\infty }\left( \Omega
\right) \cap W_{\mathrm{loc}}^{2,2}\left( \Omega \right) $ of the weak
solutions $u\in W_{\mathrm{loc}}^{1,p}\left( \Omega \right) $ to the
elliptic equation (\ref{elliptic equation}). We already observed that in the
gradient estimate (\ref{interpolation bound 2}) the exponent $\alpha $,
defined in (\ref{alpha}), is greater than or equal to $1$ and $\alpha =1$ if
and only if $q=p$. Therefore in this case the $L^{\infty }\left( \Omega
\right) -$gradient local bound (\ref{interpolation bound 2}) for $u$ takes
the form 
\begin{equation}
\left\Vert Du\left( x\right) \right\Vert _{L^{\infty }\left( B_{\varrho };%
\mathbb{R}^{n}\right) }\leq \left( \frac{c}{\left( R-\varrho \right) ^{n}}%
\int_{B_{R}}(1+\left\vert Du\left( x\right) \right\vert ^{2})^{\frac{p}{2}%
}\,dx\right) ^{\frac{1}{p}}\,,  \label{interpolation bound 2 for q=p}
\end{equation}%
while the $L^{2}\left( \Omega \right) -$local bound for the $n\times n$
matrix $D^{2}u$ of the second derivatives of $u$ is 
\begin{equation}
\int_{B_{\rho }}\left\vert D^{2}u\right\vert ^{2}\,dx\leq \frac{c}{\left(
R-\rho \right) ^{2}}\int_{B_{R}}(1+\left\vert Du\left( x\right) \right\vert
^{2})^{\frac{p}{2}}\,dx\,.  \label{bound on second derivatives for q=p}
\end{equation}

\section{Linking lemma}

As before we use the notation $a\left( x,u,\xi \right) =\left( a^{i}\left(
x,u,\xi \right) \right) _{i=1,\ldots ,n}$; moreover, as usual, $\left(
a\left( x,u,\xi \right) ,\xi \right) $ is the scalar product in $\mathbb{R}%
^{n}$ of $a\left( x,u,\xi \right) $ and $\xi $.

\begin{lemma}
\label{Lemma 1} Under the ellipticity condition (\ref{ellipticity}) and the
order-one growth conditions (\ref{growth 0-a}), (\ref{growth 0-b}) with $%
1<p\leq q<p+2$ and parameters $0\leq \alpha \leq \frac{ 2\left( q-2\right) }{%
q-p+2}$ and $0\leq \beta <\frac{n\left( p-1\right) }{n-p }:=\frac{p-1}{p}%
p^{\ast }$, the following coercivity and zero-order growth conditions hold:

(i) for some positive constants $c_{1},c_{2}$ and for all $x\in \Omega $, $%
u\in \mathbb{R}$ and $\xi \in \mathbb{R}^{n}$ 
\begin{equation}
\left( a\left( x,u,\xi \right) ,\xi \right) \geq c_{1}\left\vert \xi
\right\vert ^{p}-c_{2}\left\vert u\right\vert ^{\theta}-b_{1}\left( x\right)
\,,  \label{(i) in the Lemma}
\end{equation}
with 
\begin{equation}  \label{e:deftheta}
\theta:=\max\{\tfrac{2p}{p-q+2},\beta \tfrac{p}{p-1}\}
\end{equation}
and $b_{1}\left( x\right) :=const\cdot \left\{ 1+\left\vert a\left(
x,0,0\right) \right\vert ^{\frac{p}{p-1}}\right\} \in L_{\mathrm{loc}
}^{\gamma \frac{p-1}{p}}\,$ with $\gamma >\frac{n}{p-1}$. Recall that $%
\gamma $ is defined in (\ref{growth 3}) and note that the summability
condition here is the same as that one in the following Section \ref%
{ss:ipotesi}, with $s_{1}:=\gamma \frac{p-1}{p}$ and $L_{\mathrm{loc}%
}^{\gamma \frac{p-1}{p}}=L_{ \mathrm{loc}}^{s_{1}}$, $s_{1}>\frac{n}{p}$;

%	
%\label{Lemma 1}Under the ellipticity condition (\ref{(i) in the Section})
%and the order-one growth condition (\ref{growth 0-a}), (\ref{growth 0-b})
%with parameters $0\leq \alpha \leq \frac{2\left( q-2\right) }{q-p+2}$ and $%
%\beta \geq 0$,  {\color{blue}if the exponents $p,q$ satisfy $2\leq p\leq  q< p+2$ and  $\frac{q}{p}\le 1+\frac{2}{n}$, then} the following ellipticity and zero-order growth conditions
%hold:
%
%(i) for some positive constants $c_{1},c_{2}$ and for all $x\in \Omega $, $%
%u\in \mathbb{R}$ and $\xi \in \mathbb{R}^{n}$ 
%\begin{equation}
%\left( a\left( x,u,\xi \right) ,\xi \right) \geq c_{1}\left\vert \xi
%\right\vert ^{p}-c_{2}\left\vert u\right\vert ^{\color{blue}{\frac{2p}{p-q+2}}}-b_{1}\left( x\right) \,,  \label{(i) in the Lemma}
%\end{equation}%
%with $
%{\color{blue}
%	b_{1}\left( x\right) :=const\cdot \left\{ 1+\left\vert a\left(
%	x,0,0\right) \right\vert ^{\frac{p}{p-1}}\right\} \in L_{\mathrm{loc}%
%	}^{\gamma \frac{p-1}{p}}}$%
%and $\gamma >\frac{n}{p-1}$. Recall that $\gamma $ is defined in (\ref%
%{growth 3}) and note that the summability condition here is the same as that
%one in the following Section \ref{ss:ipotesi}, with $s_{1}:=\gamma \frac{p-1%
%}{p}$ and $L_{\mathrm{loc}}^{\gamma \frac{p-1}{p}}=L_{\mathrm{loc}}^{s_{1}}$%
%, $s_{1}>\frac{n}{p}$;

(ii) for some positive constants $c_{3},c_{4}$ and for all $x\in \Omega $, $%
u\in \mathbb{R}$, $\xi \in \mathbb{R}^{n}$ and the same $b_{1}\left(
x\right) $ 
\begin{equation}
\left\vert a\left( x,u,\xi \right) \right\vert \leq c_{3}\left\vert \xi
\right\vert ^{q-1}+c_{4}\left\vert u\right\vert ^{2\frac{q-1}{q-p+2}%
}+b_{1}\left( x\right) +1\,.  \label{(ii) in the Lemma}
\end{equation}
\end{lemma}

\begin{proof}
We start proving \textit{(i)}. We first observe that, for every $\varepsilon
>0$, by Young's inequality with conjugate exponents $\frac{p}{p-1}$ and $p$ 
\begin{equation}
\left\vert \left( a\left( x,0,0\right) ,\xi \right) \right\vert \leq
\left\vert a\left( x,0,0\right) \right\vert \left\vert \xi \right\vert \leq 
\tfrac{p-1}{p}\varepsilon ^{-\frac{p}{p-1}}\left\vert a\left( x,0,0\right)
\right\vert ^{\frac{p}{p-1}}+\frac{\varepsilon ^{p}}{p}\left\vert \xi
\right\vert ^{p}\,.  \label{consequences 0-a}
\end{equation}%
Moreover, with the component notation $\xi =\left( \xi _{i}\right)
_{i=1,\ldots ,n}$, we have 
\begin{equation}
\left( a\left( x,u,\xi \right) ,\xi \right) -\left( a\left( x,0,0\right)
,\xi \right)  \label{consequences 1}
\end{equation}%
\begin{equation*}
=\sum_{i=1}^{n}\left\{ a^{i}\left( x,u,\xi \right) -a^{i}\left( x,0,0\right)
\right\} \xi _{i}=\int_{0}^{1}\frac{d}{dt}\sum_{i=1}^{n}a^{i}\left(
x,tu,t\xi \right) \xi _{i}\,dt
\end{equation*}%
\begin{equation*}
=\int_{0}^{1}\left\{ \sum_{i=1}^{n}a_{u}^{i}\left( x,tu,t\xi \right) \xi
_{i}u+\displaystyle\sum_{i,j=1}^{n}a_{\xi _{j}}^{i}\left( x,tu,t\xi \right)
\xi _{i}\xi _{j}\right\} \,dt\,.
\end{equation*}%
By the growth condition (\ref{growth 0-b}) and the ellipticity assumption (%
\ref{ellipticity}) we get 
\begin{equation*}
\left( a\left( x,u,\xi \right) ,\xi \right) -\left( a\left( x,0,0\right)
,\xi \right)
\end{equation*}%
\begin{equation*}
\geq \int_{0}^{1}\left\{ -nM\left( (1+\left\vert t\xi \right\vert ^{2})^{%
\frac{p+q-4}{4}}\left\vert \xi \right\vert \left\vert u\right\vert
+\left\vert \xi \right\vert \left\vert u\right\vert ^{\beta }\right)
+m\left( 1+\left\vert t\xi \right\vert ^{2}\right) ^{\frac{p-2}{2}%
}\left\vert \xi \right\vert ^{2}\right\} dt
\end{equation*}%
\begin{eqnarray*}
&=&\int_{0}^{1}\left\{ -nM\left( (1+\left\vert t\xi \right\vert ^{2})^{\frac{%
p-2}{4}}\left\vert \xi \right\vert (1+\left\vert t\xi \right\vert ^{2})^{%
\frac{q-2}{4}}\left\vert u\right\vert +\left\vert \xi \right\vert \left\vert
u\right\vert ^{\beta }\right) \right. \\
&&\;\;\;\;\;\;\;\;\;\;\;\;\;\;\;\;\;\;\;\;\;\;\;\;\;\;\;\;\;\;\;\;\;\;\;\;%
\left. +\;m(1+\left\vert t\xi \right\vert ^{2})^{\frac{p-2}{2}}\left\vert
\xi \right\vert ^{2}\right\} dt
\end{eqnarray*}%
\begin{equation*}
\geq \int_{0}^{1}\left\{ -nM\left( \tfrac{\varepsilon ^{2}}{2}(1+\left\vert
t\xi \right\vert ^{2})^{\frac{p-2}{2}}\left\vert \xi \right\vert ^{2}+\tfrac{%
1}{2\varepsilon ^{2}}(1+\left\vert t\xi \right\vert ^{2})^{\frac{q-2}{2}%
}\left\vert u\right\vert ^{2}+\left\vert \xi \right\vert \left\vert
u\right\vert ^{\beta }\right) \right.
\end{equation*}%
\begin{equation*}
\;\;\;\;\;\;\;\;\;\;\;\;\;\;\;\;\;\;\;\;\;\;\;\;\;\;\;\;\;\;\;\;\;\;\;\;%
\left. +\;m(1+\left\vert t\xi \right\vert ^{2})^{\frac{p-2}{2}}\left\vert
\xi \right\vert ^{2}\right\} dt.
\end{equation*}%
For $\varepsilon >0$ sufficiently small we deduce that there exists a
constant $c>0$ such that 
\begin{equation*}
\left( a\left( x,u,\xi \right) ,\xi \right) -\left( a\left( x,0,0\right)
,\xi \right)
\end{equation*}%
\begin{equation*}
\geq \int_{0}^{1}\left\{ -c(1+\left\vert t\xi \right\vert ^{2})^{\frac{q-2}{2%
}}\left\vert u\right\vert ^{2}-nM\left\vert \xi \right\vert \left\vert
u\right\vert ^{\beta }+\frac{m}{2}(1+\left\vert t\xi \right\vert ^{2})^{%
\frac{p-2}{2}}\left\vert \xi \right\vert ^{2}\right\} dt.
\end{equation*}%
In the first addendum we take $t=1$ and we apply Young's inequality with
conjugate exponents $\frac{p}{q-2}$ and $\frac{p}{p-q+2}$ (here we use the
assumption $q<p+2$), while in the second addendum we consider the conjugate
exponents $p$ and $\frac{p}{p-1}$%
\begin{equation*}
\left( a\left( x,u,\xi \right) ,\xi \right) -\left( a\left( x,0,0\right)
,\xi \right)
\end{equation*}%
\begin{equation}
\geq -c\tfrac{q-2}{p}\varepsilon ^{\frac{p}{q-2}}(1+\left\vert \xi
\right\vert ^{2})^{\frac{p}{2}}-c\tfrac{p-q+2}{p\epsilon ^{\frac{p}{p-q+2}}}%
\left\vert u\right\vert ^{\frac{2p}{p-q+2}}  \label{e:stimaellitt1}
\end{equation}%
\begin{equation*}
-\tfrac{nM}{p}\varepsilon ^{p}\left\vert \xi \right\vert ^{p}-\tfrac{%
nM\left( p-1\right) }{p\,\varepsilon ^{\frac{p}{p-1}}}\left\vert
u\right\vert ^{\beta \frac{p}{p-1}}+\frac{m}{2}\left\vert \xi \right\vert
^{2}\int_{0}^{1}(1+\left\vert t\xi \right\vert ^{2})^{\frac{p-2}{2}}dt.
\end{equation*}%
If we \textit{\textquotedblleft just forget }$1+$\textit{"}; we have 
\begin{equation}
\frac{m}{2}\left\vert \xi \right\vert ^{2}\int_{0}^{1}(1+\left\vert t\xi
\right\vert ^{2})^{\frac{p-2}{2}}\,dt\geq \frac{m}{2}\left\vert \xi
\right\vert ^{p}\int_{0}^{1}t^{p-2}dt\,=\frac{m}{2(p-1)}\left\vert \xi
\right\vert ^{p}\,.  \label{e:stimaprinc}
\end{equation}%
Moreover, 
\begin{equation*}
|u|^{\beta \frac{p}{p-1}}+|u|^{\frac{2p}{p-q+2}}\leq \left\{ 
\begin{array}{ll}
2 & \text{if}\;\;\left\vert u\right\vert \leq 1 \\ 
2|u|^{\max \{\frac{2p}{p-q+2},\beta \frac{p}{p-1}\}} & \text{if}%
\;\;\left\vert u\right\vert >1,%
\end{array}%
\right.
\end{equation*}%
so that $|u|^{\beta \frac{p}{p-1}}+|u|^{\frac{2p}{p-q+2}}\leq 2+2\left\vert
u\right\vert ^{\theta }$, with $\theta $ defined in (\ref{e:deftheta}).
Therefore, for $\varepsilon >0$ sufficiently small, taking into account (\ref%
{consequences 0-a}),(\ref{e:stimaellitt1}),(\ref{e:stimaprinc}) we obtain
the statement in \textit{(i)}. % if 	{\color{blue}we define 
%\begin{equation}
%		b_{1}\left( x\right) :=const\cdot \left\{ 1+\left\vert a\left(
%		x,0,0\right) \right\vert ^{\frac{p}{p-1}}\right\} 
%	\label{consequences 0-b}
%\end{equation}%
%that is locally in $L^{\gamma \frac{p-1}{p}}$  by the assumption (\ref{growth 3}).} 

\medbreak

To prove (\ref{(ii) in the Lemma}) in \textit{(ii)} we\ fix $i\in \left\{
1,2,\ldots ,n\right\} $ and we consider assumptions (\ref{growth 0-a}) and (%
\ref{growth 0-b}). With the notation $\alpha =2\frac{q-2}{q-p+2}$ we get 
\begin{equation}
\left\vert a^{i}\left( x,u,\xi \right) -a^{i}\left( x,0,0\right) \right\vert
=\left\vert \,\int_{0}^{1}\frac{d}{dt}a^{i}\left( x,tu,t\xi \right)
\,dt\,\right\vert  \label{consequences 2}
\end{equation}%
\begin{equation*}
=\left\vert \,\int_{0}^{1}\left\{ a_{u}^{i}\left( x,tu,t\xi \right) u+%
\displaystyle\sum_{j=1}^{n}a_{\xi _{j}}^{i}\left( x,tu,t\xi \right) \xi
_{i}\right\} dt\,\right\vert
\end{equation*}%
\begin{equation*}
\leq M\left( (1+\left\vert \xi \right\vert ^{2})^{\frac{p+q-4}{4}}\left\vert
u\right\vert +\left\vert u\right\vert ^{\beta }\right) +nM\left(
(1+\left\vert \xi \right\vert ^{2})^{\frac{q-2}{2}+\frac{1}{2}}+\left\vert
\xi \right\vert \left\vert u\right\vert ^{\alpha }\right) \,.
\end{equation*}%
Similarly as before, we use Young's inequality with conjugate exponents $2%
\frac{q-1}{p+q-4}$ and $2\frac{q-1}{q-p+2}$ in the first addendum and with
conjugate exponents $q-1$ and $\frac{q-1}{q-2}$ in the last addendum. We
obtain 
\begin{equation}
\left\vert a^{i}\left( x,u,\xi \right) -a^{i}\left( x,0,0\right) \right\vert
\label{consequences 3}
\end{equation}%
\begin{equation*}
\leq M\left( \tfrac{1}{2}\tfrac{p+q-4}{q-1}(1+\left\vert \xi \right\vert
^{2})^{\frac{q-1}{2}}+\tfrac{1}{2}\tfrac{q-p+2}{q-1}\left\vert u\right\vert
^{2\frac{q-1}{q-p+2}}\right)
\end{equation*}%
\begin{equation*}
+nM(1+\left\vert \xi \right\vert ^{2})^{\frac{q-1}{2}}+nM\left( \tfrac{1}{q-1%
}\left\vert \xi \right\vert ^{q-1}+\tfrac{q-2}{q-1}\left\vert u\right\vert
^{\alpha \frac{q-1}{q-2}}\right) \,.
\end{equation*}%
Recall that $0\leq \alpha \leq \frac{2\left( q-2\right) }{q-p+2}$ and thus $%
0\leq \alpha \frac{q-1}{q-2}\leq \frac{2\left( q-1\right) }{q-p+2}$. Then 
\begin{equation*}
\left\vert u\right\vert ^{\alpha \frac{q-1}{q-2}}\leq \left\{ 
\begin{array}{l}
1\;\;\;\;\;\text{if}\;\;\left\vert u\right\vert \leq 1 \\ 
\left\vert u\right\vert ^{\frac{2\left( q-1\right) }{q-p+2}}\;\;\;\text{if}%
\;\;\left\vert u\right\vert >1%
\end{array}%
\right.
\end{equation*}%
and in any case 
\begin{equation}
\left\vert u\right\vert ^{\alpha \frac{q-1}{q-2}}\leq \left\vert
u\right\vert ^{\frac{2\left( q-1\right) }{q-p+2}}+1\,.
\label{consequences 4}
\end{equation}%
By combining (\ref{consequences 3}),(\ref{consequences 4}) we obtain the
conclusion (\ref{(ii) in the Lemma}).
\end{proof}

\section{Local boundedness\label{Section: Local boundedness}}

In this section we prove a local boundedness result for weak solutions to
the elliptic equation (\ref{elliptic equation}). First we recall the
definition of weak solution. In fact, in the context of $p,q-$growth
conditions, $p<q$, it is necessary to use some care in choosing the Sobolev
class where to look for solutions; for a discussion about this aspect see 
\cite[Section 3.1]{Marcellini 2023} and Remark \ref{Remark about the
summability condition on b} below. In the following Definition \ref%
{aweaksolDEF} we look for solutions in the Sobolev space $W_{\mathrm{loc}%
}^{1,q}\left( \Omega \right) $.

\begin{definition}
\label{aweaksolDEF}A function $u\in W_{\mathrm{loc}}^{1,q}(\Omega )$ is a
weak solution to (\ref{elliptic equation}) if 
\begin{equation}
\int_{\Omega }\left\{ \sum_{i=1}^{n}a^{i}(x,u,Du)\varphi
_{x_{i}}+b(x,u,Du)\varphi \right\} \,dx=0  \label{aweaksol}
\end{equation}%
for all $\varphi \in W^{1,q}(\Omega )$ with $\mathrm{supp}\,\varphi \Subset
\Omega $.
\end{definition}

In order to obtain a local boundedness result for the weak solutions of (\ref%
{elliptic equation}) we consider the assumptions of Section \ref{Section:
assumptions for the local Lipschitz continuity} and their consequences as
stated in Lemma \ref{Lemma 1}. However it could be useful to list explicitly
here the hypotheses on the Carath\'{e}odory functions $a:\Omega \times 
\mathbb{R}\times \mathbb{R}^{n}\rightarrow \mathbb{R}^{n}$ and $b:\Omega
\times \mathbb{R}\times \mathbb{R}^{n}\rightarrow \mathbb{R}$ which we use
for proving the local boundedness, assumptions which sometime are less
restrictive than those of Section \ref{Section: assumptions for the local
Lipschitz continuity} above. These \textit{coercivity} and \textit{growth
assumptions} are stated in functions of some parameters $p,q,\theta ,\delta
,s_{0},s_{1}$. Before writing the bounds of these parameters, we remind that 
$p^{\ast }$ denotes the Sobolev exponent appearing in the Sobolev embedding
theorem for functions in $W^{1,p}(\Omega )$ with $\Omega $ bounded open set
in $\mathbb{R}^{n}$; i.e. $p^{\ast }:=\frac{np}{n-p}$ if $p<n$ and $p^{\ast
} $ equal any fixed real number greater than $p$ if $p\ge n$. In particular,
if $p\ge n$ we assume $s_{0},s_{1}>1$ and, without loss of generality, we
choose 
\begin{equation}
p^{\ast }>\max \left\{\tfrac{p}{p-q+1},\theta,\delta ,\tfrac{ps_{1}}{s_{1}-1}%
,\tfrac{ps_{0}}{s_{0}-1}\right\} .  \label{e:hppstar}
\end{equation}

If $p<n$ the conditions on the exponents $p,q,\theta ,\delta ,s_{0},s_{1}$
which we consider for the local boundedness of the weak solutions are 
\begin{equation}
\text{on the ratio }\tfrac{q}{p}\geq 1\text{: }\;\;\;\;\frac{q}{p}<1+\frac{1%
}{n}\,;  \label{e:hpq}
\end{equation}%
\begin{equation}
\text{on }\theta \geq 0\text{ and }\delta \geq 1\text{: }\;\;\;\;\theta
,\delta <\frac{np}{n-p}=p^{\ast }\,;  \label{e:stimaIIesponente}
\end{equation}%
\begin{equation}
\text{on }s_{0}\text{ and }s_{1}\text{: }\;\;\;\;s_{0},s_{1}>\frac{n}{p}\,.
\label{e:hpgamma}
\end{equation}

\subsection{Assumptions for the local boundedness\label{ss:ipotesi}}

We start by \textit{(i),(ii)} in Lemma \ref{Lemma 1}, in this slightly less
restrictive form (and, with abuse of notation, by changing $b_{1}\left(
x\right) +1$ with $b_{1}\left( x\right) $):

\textit{(i)} there exist $p,q$, $1\leq p\leq q<p+1$ and $\theta \geq 0$ and
some positive constants $c_{1},c_{2}$, such that for a.e. $x\in \Omega $ and
every $u\in \mathbb{R}$ and $\xi \in \mathbb{R}^{n}$ 
\begin{equation}
\left( a\left( x,u,\xi \right) ,\xi \right) \geq c_{1}\left\vert \xi
\right\vert ^{p}-c_{2}\left\vert u\right\vert ^{\theta }-b_{1}\left(
x\right) \,,  \label{(i) in the Section}
\end{equation}%
with $b_{1}\geq 0$ and $b_{1}\in L_{\mathrm{loc}}^{s_{1}}(\Omega )$ for some 
$s_{1}>\frac{n}{p}$ (here $s_{1}$ has the role stated in Lemma \ref{Lemma 1}%
: $s_{1}:=\gamma \frac{p-1}{p}>\frac{n}{p}$);

\textit{(ii)} for some positive constants $c_{3},c_{4}$, for a.e. $x\in
\Omega $, every $u\in \mathbb{R}$ and $\xi \in \mathbb{R}^{n}$ 
\begin{equation}
\left\vert a\left( x,u,\xi \right) \right\vert \leq c_{3}\left\vert \xi
\right\vert ^{q-1}+c_{4}\left\vert u\right\vert ^{\left( q-1\right) \frac{2}{%
q-p+2}}+b_{1}\left( x\right) \,  \label{(ii) in the Section}
\end{equation}%
with the same $b_{1}\in L_{\mathrm{loc}}^{s_{1}}(\Omega ) $ as in (\ref{(i)
in the Section});

\textit{(iii)} for the same exponents $p,q$, for $\delta \geq 1$ and for a
positive constant $M$ 
\begin{equation}
\left\vert b\left( x,u,\xi \right) \right\vert \leq M\left( 1+\left\vert \xi
\right\vert ^{2}\right) ^{\frac{p+q-2}{4}}+M\left\vert u\right\vert ^{\delta
-1}+b_{0}(x)  \label{b locally bounded in section}
\end{equation}%
with $b_{0}\ge 0$ and $b_{0}\in L_{\mathrm{loc}}^{s_{0}}(\Omega )$, with $%
s_{0}>\frac{n}{p}$.

\subsection{Statement of the local boundedness result}

The next local boundedness Theorem \ref{t:mainboundeness}, valid for weak
solutions to the differential equation (\ref{elliptic equation}), will be
used also in the proof of the local Lipschitz continuity Theorem \ref{final
theorem}.

\begin{theorem}
\label{t:mainboundeness} Let $u\in W_{\mathrm{loc}}^{1,q}(\Omega )$ be a
weak solution to (\ref{elliptic equation}) under the assumptions in Section %
\ref{ss:ipotesi} and (\ref{e:hppstar})--(\ref{e:hpgamma}). Consider $%
0<R_{0}\leq 1$ with $B_{R_{0}}(x_{0})\Subset \Omega $. Then there exists $%
\sigma >0$ such that 
\begin{equation}
\Vert u\Vert _{L^{\infty }(B_{r/2}(x_{0}))}\leq \frac{c}{r^{\frac{p}{\sigma
(p-q+1)}}}\left( 1+\Vert u\Vert _{L^{p^{\ast }}( B_r (x_{0}))}\right) ^{%
\frac{p^{\ast }-p}{\sigma }}  \label{e:stimalim}
\end{equation}%
for every positive $r\leq R_{0}$, where the constant $c$ depends on the $%
L^{s_{0}}$-norm of $b_{0}$ and the $L^{s_{1}}$-norm of $b_{1}$ in $B_{R_{0}}$
and it is independent of $u$.
\end{theorem}

The explicit expression for $\sigma $ in (\ref{e:stimalim}) is 
\begin{equation}
\sigma :=p^{\ast }-\max \left\{ \tfrac{p}{p-q+1};\;\tfrac{p^{\ast }}{s_{1}}%
+1;\;\tfrac{p^{\ast }}{s_{0}}+1;\;\theta ;\;\delta \right\}
\label{e:defsigma}
\end{equation}%
and we note that $\sigma >0$; this is due to the bound (\ref{e:hppstar}) if $%
p\geq n$, and to the bounds (\ref{e:hpq}),(\ref{e:stimaIIesponente}),(\ref%
{e:hpgamma}) if $p<n$; in particular we notice that $\tfrac{p}{p-q+1}%
<p^{\ast }$ if and only if $\frac{q}{p}<1+\frac{1}{n}$, i.e. (\ref{e:hpq})
holds.

\begin{remark}[about the summability condition on $b$]
\label{Remark about the summability condition on b}We discuss the
summability of the integral form (\ref{aweaksol}) of the elliptic equation (%
\ref{elliptic equation}), in particular the correctness of the definition of
the pairing 
\begin{equation}
\int_{\Omega }b(x,u,Du)\,\varphi \left( x\right) \,dx\,.
\label{pairing with b}
\end{equation}%
Since $\varphi $ is a test function in $W^{1,q}\left( \Omega \right) $ with
compact support in $\Omega $, for the right hand side $b$, satisfying the
bound in (\ref{b locally bounded in section}), for the $x-$dependence it is
natural to assume the summability $b_{0}\in L_{\mathrm{loc}}^{q\prime
}\left( \Omega \right) $, with as usual $\frac{1}{q}+\frac{1}{q^{\prime }}=1$%
. This summability $b_{0}\in L_{\mathrm{loc}}^{q\prime }\left( \Omega
\right) $ is a consequence of the assumption in (\ref{b locally bounded}) $%
b_{0}\in L_{\mathrm{loc}}^{s_{0}}\left( \Omega \right) $ with $s_{0}>n$; in
fact, since $q\geq 2$, $s_{0}>n\geq 2\geq \frac{q}{q-1}:=q^{\prime }$.
However the local boundedness result in Theorem \ref{t:mainboundeness} is
obtained with the less restrictive bound in (\ref{b locally bounded in
section}) $b_{0}\in L_{\mathrm{loc}}^{s_{0}}\left( \Omega \right) $ with $%
s_{0}>\frac{n}{p}$. Therefore in this Section still we need to show that the
pairing (\ref{pairing with b}) is well defined. This fact is a consequence
of the imbedding of $W_{0}^{1,q}\left( \Omega \right) $ into $L^{q^{\ast
}}\left( \Omega \right) $; in fact the test function $\varphi \in L^{q^{\ast
}}\left( \Omega \right) $ and the integral (\ref{pairing with b}) is also
correctly defined if $b_{0}\in L_{\mathrm{loc}}^{(q^{\ast })^{\prime
}}\left( \Omega \right) $. Let us show that $b_{0}\in L_{\mathrm{loc}%
}^{s_{0}}\left( \Omega \right) \subset L_{\mathrm{loc}}^{(q^{\ast })^{\prime
}}\left( \Omega \right) $ when $s_{0}>\frac{n}{p}$ and $1\leq p<n$. Since $%
p\leq q$ then $p^{\ast }\leq q^{\ast }$ and 
\begin{equation*}
(q^{\ast })^{\prime }\leq (p^{\ast })^{\prime }=\frac{p^{\ast }}{p^{\ast }-1}%
=\frac{np}{np-n+p};
\end{equation*}%
this last quantity is less than or equal to $\frac{n}{p}$ if and only if $%
\frac{p}{np-n+p}\leq \frac{1}{p}$, which is equivalent to $p^{2}-\left(
n+1\right) p+n\leq 0$; i.e. $1\leq p\leq n$. Therefore $(q^{\ast })^{\prime
}\leq \frac{n}{p}<s_{0}$ and $L_{\mathrm{loc}}^{s_{0}}\left( \Omega \right)
\subset L_{\mathrm{loc}}^{(q^{\ast })^{\prime }}\left( \Omega \right) $.

If $p=n$ the inclusion $L_{\mathrm{loc}}^{s_{0}}\left( \Omega \right)
\subset L_{\mathrm{loc}}^{(q^{\ast })^{\prime }}\left( \Omega \right) $ is
trivially satisfied by assuming $q^*$ large enough, in dependence of $s_0>1$%
. If $p>n$ the test functions are bounded, therefore the pairing is well
defined because $b_0$ is locally summable.

As far as the $u-$dependence it is concerned, we remark that the definition
of the pairing is correct if $|u|^{\delta-1}\varphi\in L^1(\Omega)$. Since,
by Sobolev embedding Theorem $\varphi $ is in $L^{q^*}\left( \Omega \right) $
with compact support in $\Omega $, then we need $u\in L_{\mathrm{loc}%
}^{(\delta-1)\frac{q^*}{q^*-1}}(\Omega)$. By Sobolev embedding Theorem $u\in
L^{q^*}\left( \Omega \right)$, therefore the needed summability for $u$ is
satisfied if $(\delta-1)\frac{q^*}{q^*-1}\le q^*$, or equivalently, $%
\delta\le q^*$. This last condition holds, because $\delta< p^*$, see (\ref%
{e:stimaIIesponente}).

We conclude by considering the product $\left\vert Du \right\vert^{\frac{%
p+q-2}{2}}\varphi$ that we want to be summable in $\Omega$. Reasoning as
above, this happens if $\frac{p+q-2}{2}\frac{q^*}{q^*-1}\le q$. If $q\ge n$
it is sufficient to choose $q^*$ large enough, since $p<q+2$. If instead $q<
n$ then we observe that $\frac{p+q-2}{2}\frac{q^*}{q^*-1}\le q$ is
consequence of $p\le q$.
\end{remark}

\begin{remark}[about an upper bound of $\protect\theta $]
\label{r:theta<p*} In Section \ref{Section: Local Lipschitz continuity} we
will define $\theta $ as in (\ref{e:deftheta}). It is easy to check that if $%
\frac{q}{p}<1+\frac{2}{n}$, thus in particular if (\ref{e:hpq}) holds, then $%
\tfrac{2p}{p-q+2}<p^{\ast }$. Therefore the conditions $\frac{q}{p}<1+\frac{2%
}{n}\ \ $and$\ \ \beta <\frac{n\left( p-1\right) }{n-p}:=\frac{p-1}{p}%
p^{\ast }$ imply $\theta :=\max \{\tfrac{2p}{p-q+2},\beta \tfrac{p}{p-1}%
\}<p^{\ast }$.
\end{remark}

\subsection{The Caccioppoli's inequality}

\label{ss:caccio}

We prove a Caccioppoli's inequality for the weak solutions of (\ref{elliptic
equation}) under the assumptions stated in Section \ref{ss:ipotesi}.

\begin{proposition}
Let $u\in W_{\mathrm{loc}}^{1,q}(\Omega )$ be a weak solution to (\ref%
{elliptic equation}) under the assumptions in Section \ref{ss:ipotesi}.
Consider $B_{R_{0}}(x_{0})\Subset \Omega $ and for every $k\in \mathbb{R}$, $%
k\geq 0$, and every $R\leq R_{0}$ consider the super-level sets $%
A_{k,R}:=\{x\in B_{R}(x_{0})\,:\,u(x)>k\}$. Then there exists $c$ depending
only on the data, but neither on $u$ nor $k$, such that for every $\rho ,R$
such that $0< \rho <R\leq R_{0}\leq 1$, 
\begin{align}
& \int_{A_{k,\rho }}|Du|^{p}\,dx\leq \tfrac{c}{(R-\rho )^{\frac{p}{p-q+1}}}%
\int_{A_{k,R}}(u-k)^{\frac{p}{p-q+1}}\,dx  \notag \\
& +\,\tfrac{c}{R-\rho }\Vert b_{1}\Vert _{L^{s_{1}}(B_{R_{0}})}\Vert
u-k\Vert _{L^{p^{\ast }}(A_{k,R})}|A_{k,R}|^{1-\frac{1}{s_{1}}-\frac{1}{%
p^{\ast }}}  \notag \\
& +\,c\Vert b_{0}+1\Vert _{L^{s_{0}}(B_{R_{0}})}\Vert u-k\Vert _{L^{p^{\ast
}}(A_{k,R})}|A_{k,R}|^{1-\frac{1}{s_{0}}-\frac{1}{p^{\ast }}}  \notag \\
& +c\,\Vert u-k\Vert _{L^{p^{\ast }}(A_{k,R})}^{\frac{2p}{q-p+2}%
}|A_{k,R}|^{1-\frac{2p}{p^{\ast }(q-p+2)}}+c\,\Vert u-k\Vert _{L^{p^{\ast
}}(A_{k,R})}^{\theta }|A_{k,R}|^{1-\frac{\theta }{p^{\ast }}}  \notag \\
& +c\,\Vert u-k\Vert _{L^{p^{\ast }}(A_{k,R})}^{\delta }|A_{k,R}|^{1-\frac{%
\delta }{p^{\ast }}}+c\,(1+k^{\tau })|A_{k,R}|+c\Vert b_{1}\Vert
_{L^{s_{1}}(B_{R_{0}})}|A_{k,R}|^{1-\frac{1}{s_{1}}},  \label{e:fineIIIstep}
\end{align}%
where 
\begin{equation}
\tau :=\max \left\{ \tfrac{2p}{q-p+2},\theta ,\delta \right\} .
\label{e:expkappa}
\end{equation}
\end{proposition}

\begin{proof}
We split the proof into steps.

\textit{Step 1. }Consider $B_{R_{0}}(x_{0})\Subset \Omega $, $0< \rho <R\leq
R_{0}\leq 1$. Let $\eta \in C_{0}^{\infty }(B_{R})$ be a cut-off function,
satisfying the following assumptions: 
\begin{equation}
0\leq \eta \leq 1,\quad \eta \equiv 1\ \text{in $B_{\rho }(x_{0})$,}\quad 
\text{$|D\eta |\leq \tfrac{2}{R-\rho }$.}  \label{eta}
\end{equation}%
For every $k\geq 1$ we define the test function $\varphi _{k}$ as follows 
\begin{equation*}
\varphi _{k}(x):=(u(x)-k)_{+}[\eta (x)]^{\mu }\quad \text{for a.e.}\ x\in
B_{R_{0}}(x_{0}),
\end{equation*}%
where $(u(x)-k)_{+}=\max \{u(x)-k,0\}$ and 
\begin{equation}
\mu :=\tfrac{p}{p-q+1}.  \label{definizione-mu}
\end{equation}%
Notice that $\mu $ is greater than $1$ because $q>1$ and that $\varphi
_{k}\in W_{0}^{1,q}(B_{R_{0}}(x_{0}))$, $\operatorname{supp}\varphi _{k}\Subset
B_{R}(x_{0})$.

% 
% We define 
% $\tilde{a}=(\tilde{a}^{\alpha}_{i})_{i=1,...,n}$
% \begin{equation}\label{tildea}
% \tilde{a}_i(x,u,\xi)=\sum_{j=1}^na_{ij}\left(
% x,u,\xi\right) \xi_{j}+b_{i}\left(
% x,u,\xi\right)
% \end{equation}
%and we used the following notation: $u\otimes
%D\eta:=\left(u^{\alpha}\eta_{x_i}\right)_{\genfrac{}{}{0pt}{3}{i=1,...,n}{\alpha=1,...,m}}$.
\textit{Step 2. }Let us consider the super-level sets: $A_{k,R}:=\{x\in
B_{R}(x_{0})\,:\,u(x)>k\}$. Using $\varphi _{k}$ as a test function in (\ref%
{aweaksol}) we get 
\begin{equation}
\begin{aligned} I_1:=& \int_{A_{k,R}} \langle a(x,u,Du),Du\rangle
\,\eta^{\mu}\,dx \\=& -\mu \int_{A_{k,R}} \langle a(x,u,Du), D\eta\rangle
\eta^{\mu-1}(u-k)\,dx \\&- \int_{A_{k,R}} b(x,u,Du) (u-k)\,\eta^{\mu}\,dx
=:I_2+I_3. \end{aligned}  \label{NEWaveroiniziostep1}
\end{equation}%
Now, we separately consider and estimate $I_{i}$, $i=1,2,3$.

%
%\medbreak

%{\sc Estimate of $I_1$}

%%\noindent
%For a.e. $x\in \{|Du|\le 1\}$
%\begin{align}\nonumber & |\langle b(x,u,Du),Du\rangle|  \le \Lambda
%\left\{|Du|^{(p-1)(1-\epsilon)}+b_0|u|^{\frac{\gamma}{p'}}
%+a_2(x)\right\}|Du| \\	  & \le \Lambda
%\left\{1+b_0|u|^{\frac{\gamma}{p'}}
%+a_2(x)\right\}
%\le \label{b}
%c
%\left\{b_0^{p'}(x)|u|^{\gamma}+a_2^{p'}(x)
%\right\}\qquad \text{a.e. in $\{|Du|< 1\}$}
%\end{align}

%Letus now consider  the set $\{|Du|\ge 1\}$. For a.e. $x\in \{|Du|\ge 1\}$, by (\ref{growth-newb}) the  Young inequality (applied first with
%exponent $p$ and then with exponent $\frac{1}{1-\epsilon}$) and
%  (\ref{(i) in the Section})$_1$  we have that for some positive $\tau<\!\!<1$
%\begin{align}\nonumber 
%&|\langle b(x,u,Du),Du\rangle|\le \Lambda
%\left\{|Du|^{(p-1)(1-\epsilon)}+b_0|u|^{\frac{\gamma}{p'}}
%+a_2(x)\right\}|Du| 	 \\ & \le \tau|Du|^p+c_{\tau}
%\left\{|Du|^{p(1-\epsilon)}+b_0^{p'}(x)|u|^{\gamma}
%+a_2^{p'}(x)\right\}\nonumber \\ &\nonumber \le
%\frac{\lambda}{4n^{p-1}}|Du|^{p}+c_1\left\{b_0^{p'}(x)|u|^{\gamma}
%+a_2^{p'}(x)+1\right\}
%\\&\le \frac{1}{4}\langle a(x,u,Du),Du\rangle+
%c_2\left\{b_0^{p'}(x)|u|^{\gamma}
%+a_2^{p'}(x)\right\}\qquad \text{a.e. in $\{|Du|\ge 1\}$}\label{e:IIstimab}\end{align} with $c_2$
%depending also on $\epsilon$. Thus,
%\begin{equation}\label{I_1I}
%I_1 \ge \frac{3}{4}\int_{A_{k,R}\cap \{|Du|\ge 1\}} \langle
%a(x,u,Du),Du\rangle\eta^{\mu}\,dx -c_2
%\int_{A_{k,R}}\left\{b_0^{p'}|u|^{\gamma}+a_2^{p'}\right\}\,dx.
% \end{equation}

\medbreak

\textsc{Estimate of $I_{3}.$} By (\ref{b locally bounded in section}) there
exists $c(M,p,q)$ positive constant such that 
\begin{equation*}
\left\vert b\left( x,u,\xi \right) \right\vert \leq c(M,p,q)\left\{
\left\vert \xi \right\vert ^{\frac{p+q-2}{2}}+\left\vert u\right\vert
^{\delta -1}+b_{0}(x)+1\right\};
\end{equation*}%
therefore 
\begin{equation*}
I_{3}\leq c(M,p,q)\int_{A_{k,R}}\eta ^{\mu }\left\{ \left\vert Du\right\vert
^{\frac{p+q-2}{2}}(u-k)+|u|^{\delta -1}(u-k)+(b_{0}+1)(u-k)\right\} \,dx.
\end{equation*}%
We estimate the right-hand side using the Young's inequality, with exponents 
$\frac{2p}{p+q-2}$ and $\frac{2p}{p-q+2}$. There exists $c>0$, such that 
\begin{equation*}
c(M,p,q)\left\vert Du\right\vert ^{\frac{p+q-2}{2}}(u-k)\leq \frac{c_{1}}{4}%
|Du|^{p}+c(u-k)^{\frac{2p}{p-q+2}}\qquad \text{a.e. in $A_{k,R}$,}
\end{equation*}%
where $c_{1}$ is the constant in (\ref{(i) in the Section}). 
% For a.e.  $x\in \{|Du|\ge 1\}$  we get
% \[
% \Lambda|Du|^{r}(u-k)
%\le
% \Lambda|Du|^{r}(u-k).\]
Therefore 
\begin{align}
I_{3}\leq & \tfrac{c_{1}}{4}\int_{A_{k,R}}|Du|^{p}\eta ^{\mu }\,dx  \notag \\
& +c\int_{A_{k,R}}\eta ^{\mu }\left\{ (u-k)^{\frac{2p}{p-q+2}}+|u|^{\delta
-1}(u-k)\right\} \,dx  \notag \\
& +c\int_{A_{k,R}}\eta ^{\mu }(b_{0}+1)(u-k)\,dx.  \label{I_4I}
\end{align}%
Collecting (\ref{NEWaveroiniziostep1}),(\ref{I_4I}) and using (\ref{(i) in
the Section}) we get 
\begin{align}
& \tfrac{3c_{1}}{4}\int_{A_{k,R}}|Du|^{p}\eta ^{\mu }\,dx\leq I_{2}  \notag
\\
& +c\int_{A_{k,R}}\left\{ |u|^{\theta }+(u-k)^{\frac{2p}{p-q+2}}+|u|^{\delta
-1}(u-k)+(b_{0}+1)(u-k)+b_{1}\right\} \,dx.  \label{primadiI_3}
\end{align}

\textsc{Estimate of $I_{2}$}. \ For a.e. $x\in A_{k,R}\cap \{\eta \neq 0\}$
we have, by (\ref{(ii) in the Section}), 
\begin{equation*}
|\langle a(x,u,Du),D\eta \rangle |\leq |D\eta |\left(
c_{3}|Du|^{q-1}+c\left\vert u\right\vert ^{2\frac{q-1}{q-p+2}%
}+c\,b_{1}\right) ,
\end{equation*}%
where $c_{3}$ is the constant in (\ref{(ii) in the Section}). For a.e. $x\in
A_{k,R}\cap \{\eta \neq 0\}$, by $q<p+1$, the Young's inequality with
exponents $\frac{p}{q-1}$ and $\frac{p}{p-q+1}$, and noting that, by (\ref%
{definizione-mu}), $\mu -1=\mu \frac{q-1}{p}$, we get 
\begin{equation*}
\mu c_{3}|Du|^{q-1}|D\eta |(u-k)\eta ^{\mu -1}\leq \frac{c_{1}}{4}%
|Du|^{p}\eta ^{\mu }+c\mu ^{\frac{p}{p-q+1}}|D\eta |^{\frac{p}{p-q+1}}(u-k)^{%
\frac{p}{p-q+1}}.
\end{equation*}%
Therefore, 
\begin{align*}
I_{2}\leq & \tfrac{c_{1}}{4}\int_{A_{k,R}}|Du|^{p}\eta ^{\mu }\,dx+c\mu ^{%
\frac{p}{p-q+1}}\int_{A_{k,R}}|D\eta |^{\frac{p}{p-q+1}}(u-k)^{\frac{p}{p-q+1%
}}\,dx \\
& +c\int_{A_{k,R}}|D\eta |\eta ^{\mu -1}\left\{ \left\vert u\right\vert ^{2%
\frac{q-1}{q-p+2}}+b_{1}\right\} (u-k)\,dx.
\end{align*}%
By (\ref{primadiI_3}) and the inequality above, we get 
\begin{align}
& \tfrac{c_{1}}{2}\int_{A_{k,R}}|Du|^{p}\eta ^{\mu }\,dx\leq
c\int_{A_{k,R}}|D\eta |^{\frac{p}{p-q+1}}(u-k)^{\frac{p}{p-q+1}}\,dx  \notag
\\
& +c\int_{A_{k,R}}|D\eta |\left\vert u\right\vert ^{2\frac{q-1}{q-p+2}%
}(u-k)\,dx+c\int_{A_{k,R}}|u|^{\delta -1}(u-k)\,dx  \notag \\
& +c\int_{A_{k,R}}\left\{ |u|^{\theta }+(u-k)^{\frac{2p}{p-q+2}}\right\} \,dx
\notag \\
& +c\int_{A_{k,R}}|D\eta
|b_{1}(u-k)\,dx+c\int_{A_{k,R}}(b_{0}+1)(u-k)\,dx+c\int_{A_{k,R}}b_{1}\,dx.
\label{e:radice}
\end{align}%
We have $|u|^{2\frac{q-1}{q-p+2}}(u-k)_{+}\leq c\,\left( |u-k|^{2\frac{q-1}{%
q-p+2}+1}+k^{2\frac{q-1}{q-p+2}}(u-k)_{+}\right) $ for some positive $c$
depending only on $p$ and $q$. By Young inequality with exponents $\frac{p}{%
q-1}$ and $\frac{p}{p-q+1}$, we obtain 
\begin{align*}
& \int_{A_{k,R}}|D\eta |(u-k)^{2\frac{q-1}{q-p+2}+1}\,dx \\
\leq & \ c\int_{A_{k,R}}|D\eta |^{\frac{p}{p-q+1}}(u-k)^{\frac{p}{p-q+1}%
}\,dx+c\int_{A_{k,R}}(u-k)^{\frac{2p}{q-p+2}}\,dx.
\end{align*}%
Analogously, 
\begin{equation*}
\int_{A_{k,R}}|D\eta |k^{2\frac{q-1}{q-p+2}}(u-k)\,dx\leq \
c\int_{A_{k,R}}|D\eta |^{\frac{p}{p-q+1}}(u-k)^{\frac{p}{p-q+1}%
}\,dx+c\int_{A_{k,R}}k^{\frac{2p}{q-p+2}}\,dx.
\end{equation*}%
So we get 
\begin{align}
& \int_{A_{k,R}}|D\eta ||u|^{2\frac{q-1}{q-p+2}}(u-k)\,dx  \notag \\
& \leq c\int_{A_{k,R}}|D\eta |^{\frac{p}{p-q+1}}(u-k)^{\frac{p}{p-q+1}%
}\,dx+c\int_{A_{k,R}}(u-k)^{\frac{2p}{q-p+2}}\,dx+c\int_{A_{k,R}}k^{\frac{2p%
}{q-p+2}}\,dx.  \label{e:|u|u1}
\end{align}%
For a.e. $x\in A_{k,R}$ $|u|^{\delta -1}(u-k)_{+}\leq u^{\delta }\leq
c(u-k)^{\delta }+ck^{\delta }$ and we get 
\begin{equation}
\int_{A_{k,R}}|u|^{\delta -1}(u-k)\,dx\leq c\int_{A_{k,R}}\{(u-k)^{\delta
}+k^{\delta }\}\,dx.  \label{e:|u|u2}
\end{equation}%
Analogously, for a.e. $x\in A_{k,R}$ we have $|u|^{\theta }\leq
c(u-k)^{\theta }+ck^{\theta }$; therefore 
\begin{equation}
\int_{A_{k,R}}|u|^{\theta }\,dx\leq c\int_{A_{k,R}}(u-k)^{\theta
}\,dx+ck^{\theta }|A_{k,R}|.  \label{e:ubeta+1}
\end{equation}%
%
%
%
%
%
%
%
%
%
%
%
%
%
%
%
%
%
%
%
%
%
%
%
%
%
%
%
%
%
%
%
%
%
%
%
%
%
%
%
%	
%	
%	
%\begin{equation}
%|u|^{(\beta+1)\frac{p}{p-1}}\le c\int_{A_{k,R}}(u-k)^{(\beta+1)\frac{p}{p-1}}\,dx+ck^{(\beta+1)%
%\frac{p}{p-1}}|A_{k,R}|.  \label{e:ubeta+1}
%\end{equation}
By H\"{o}lder inequality with exponents $s_{1}$ and $\frac{s_{1}}{s_{1}-1}$,
we get 
\begin{equation*}
\int_{A_{k,R}}|D\eta |b_{1}(u-k)\,dx\leq c\left( \int_{A_{k,R}}|D\eta |^{%
\frac{s_{1}}{s_{1}-1}}(u-k)^{\frac{s_{1}}{s_{1}-1}}\,dx\right) ^{1-\frac{1}{%
s_{1}}}\Vert b_{1}\Vert _{L^{s_{1}}(B_{R_{0}})}.
\end{equation*}%
Since $u\in L^{p^{\ast }}(B_{R_{0}})$ and $\frac{s_{1}}{s_{1}-1}<p^{\ast }$,
by using the H\"{o}lder inequality with exponent $p^{\ast }\frac{s_{1}-1}{%
s_{1}}$ and (\ref{eta}), 
\begin{equation*}
\left( \int_{A_{k,R}}|D\eta |^{\frac{s_{1}}{s_{1}-1}}(u-k)^{\frac{s_{1}}{%
s_{1}-1}}\,dx\right) ^{1-\frac{1}{s_{1}}}\leq \tfrac{c}{R-\rho }\Vert
u-k\Vert _{L^{p^{\ast }}(A_{k,R})}|A_{k,R}|^{1-\frac{1}{s_{1}}-\frac{1}{%
p^{\ast }}}.
\end{equation*}%
We conclude that 
\begin{equation}
\int_{A_{k,R}}|D\eta |b_{1}(u-k)\,dx\leq \tfrac{c}{R-\rho }\Vert b_{1}\Vert
_{L^{s_{1}}(B_{R_{0}})}\Vert u-k\Vert _{L^{p^{\ast }}(A_{k,R})}|A_{k,R}|^{1-%
\frac{1}{s_{1}}-\frac{1}{p^{\ast }}}.  \label{e:b1u-k}
\end{equation}%
Analogously, by the integrability assumption on $b_{0}$, $b_{0}\in
L^{s_{0}}(B_{R_{0}})$, and, by H\"{o}lder inequality, 
\begin{equation}
\int_{A_{k,R}}(b_{0}+1)(u-k)\,dx\leq c\,\Vert b_{0}+1\Vert
_{L^{s_{0}}(B_{R_{0}})}\Vert u-k\Vert _{L^{p^{\ast }}(A_{k,R})}|A_{k,R}|^{1-%
\frac{1}{s_{0}}-\frac{1}{p^{\ast }}}.  \label{e:b2u-k}
\end{equation}%
Of course, 
\begin{equation}
\int_{A_{k,R}}b_{1}\,dx\leq \Vert b_{1}\Vert
_{L^{s_{1}}(B_{R_{0}})}|A_{k,R}|^{1-\frac{1}{s_{1}}}.  \label{e:b1}
\end{equation}%
Let us denote $\tau $ as in (\ref{e:expkappa}); i.e. $\tau :=\max \left\{ 
\tfrac{2p}{q-p+2},\theta ,\delta \right\} ;$ then $k^{\theta }+k^{\frac{2p}{%
q-p+2}}+k^{\delta }\leq 3(1+k^{\tau })$. Collecting this last inequality, (%
\ref{e:radice})-(\ref{e:b1}) and using (\ref{eta}) we obtain 
\begin{align}
& \int_{A_{k,\rho }}|Du|^{p}\,dx\leq \tfrac{c}{(R-\rho )^{\frac{p}{p-q+1}}}%
\int_{A_{k,R}}(u-k)^{\frac{p}{p-q+1}}\,dx  \notag \\
& +\,c\,\tfrac{1}{R-\rho }\Vert b_{1}\Vert _{L^{s_{1}}(B_{R_{0}})}\Vert
u-k\Vert _{L^{p^{\ast }}(A_{k,R})}|A_{k,R}|^{1-\frac{1}{s_{1}}-\frac{1}{%
p^{\ast }}}  \notag \\
& +\,c\,\Vert b_{0}+1\Vert _{L^{s_{0}}(B_{R_{0}})}\Vert u-k\Vert
_{L^{p^{\ast }}(A_{k,R})}|A_{k,R}|^{1-\frac{1}{s_{0}}-\frac{1}{p^{\ast }}} 
\notag \\
& +c\int_{A_{k,R}}\left( (u-k)^{\frac{2p}{q-p+2}}+(u-k)^{\theta
}+(u-k)^{\delta }\right) \,dx  \notag \\
& +c\,(1+k^{\tau })|A_{k,R}|+c\Vert b_{1}\Vert
_{L^{s_{1}}(B_{R_{0}})}|A_{k,R}|^{1-\frac{1}{s_{1}}}.  \label{e:Istep}
\end{align}

\textit{Step 3. }Consider 
\begin{equation*}
J:=c\int_{A_{k,R}}\left( (u-k)^{\frac{2p}{q-p+2}}+(u-k)^{\theta
}+(u-k)^{\delta }\right) \,dx,
\end{equation*}%
that is the last integral at the right hand side of (\ref{e:Istep}). We note
that $\frac{2p}{q-p+2}\leq \frac{2p}{2}=p$. By the definition of $p^{\ast }$
and of $\tau $, given in (\ref{e:expkappa}), by the assumptions (\ref%
{e:hppstar}),(\ref{e:hpq}),(\ref{e:stimaIIesponente}) we conclude that $\tau
:=\max \left\{ \tfrac{2p}{q-p+2};\theta ;\delta \right\} <p^{\ast }$.
Therefore, the H\"{o}lder inequality with exponent $p^{\ast }\frac{q-p+2}{2p}
$ implies 
\begin{equation*}
\int_{A_{k,R}}(u-k)^{\frac{2p}{q-p+2}}\,dx\leq \Vert u-k\Vert _{L^{p^{\ast
}}(A_{k,R})}^{\frac{2p}{q-p+2}}|A_{k,R}|^{1-\frac{2p}{p^{\ast }(q-p+2)}}.
\end{equation*}%
The H\"{o}lder inequality with exponent $\frac{p^{\ast }}{\theta }$ gives 
\begin{equation*}
\int_{A_{k,R}}(u-k)^{\theta }\,dx\leq \Vert u-k\Vert _{L^{p^{\ast
}}(A_{k,R})}^{\theta }|A_{k,R}|^{1-\frac{\theta }{p^{\ast }}}
\end{equation*}%
and, using H\"{o}lder inequality with exponent $\frac{p^{\ast }}{\delta }$,
we get 
\begin{equation*}
\int_{A_{k,R}}(u-k)^{\delta }\,dx\leq \Vert u-k\Vert _{L^{p^{\ast
}}(A_{k,R})}^{\delta }|A_{k,R}|^{1-\frac{\delta }{p^{\ast }}}\,.
\end{equation*}%
Therefore 
\begin{align*}
J\leq & c\,\Vert u-k\Vert _{L^{p^{\ast }}(A_{k,R})}^{\frac{2p}{q-p+2}%
}|A_{k,R}|^{1-\frac{2p}{p^{\ast }(q-p+2)}}+c\,\Vert u-k\Vert _{L^{p^{\ast
}}(A_{k,R})}^{\theta }|A_{k,R}|^{1-\frac{\theta }{p^{\ast }}} \\
& +c\,\Vert u-k\Vert _{L^{p^{\ast }}(A_{k,R})}^{\delta }|A_{k,R}|^{1-\frac{%
\delta }{p^{\ast }}}.
\end{align*}%
By this estimate, together with (\ref{e:Istep}), we conclude that
Caccioppoli's inequality (\ref{e:fineIIIstep}) holds.
\end{proof}

\subsection{The recursive formula}

Now we proceed towards the proof of our Theorem \ref{t:mainboundeness} by
setting up the celebrated De Giorgi's iterative method. In what follows, we
tacitly understand that all the assumptions and the no\-ta\-tion
in\-tro\-duced in the previous subsections do apply.

Let $u\in W_{\mathrm{loc}}^{1,q}(\Omega )$ be a weak solution to (\ref%
{elliptic equation}). Fix a point $x_{0}\in \Omega $ and a real $R_{0}\in
(0,1]$ in such a way that 
\begin{equation}
B_{R_{0}}(x_{0})\Subset \Omega .  \label{eq.choiceR0less1}
\end{equation}%
For every fixed $R\in (0,R_{0}]$, we then define the following (decreasing)
sequences: 
\begin{equation}
\rho _{h}:=\frac{R}{2}\left( 1+\frac{1}{2^{h}}\right) \quad \text{and}\quad 
\overline{\rho }_{h}:=\frac{\rho _{h+1}+\rho _{h}}{2},\qquad h\in \mathbb{N}%
\cup \{0\}.  \label{eq.defrhohrhobar}
\end{equation}%
Moreover, given any real number $d\geq 1$, we consider the (increasing)
sequence 
\begin{equation}
k_{h}:=d\left( 1-\frac{1}{2^{h+1}}\right) ,\qquad h\in \mathbb{N}\cup \{0\}.
\label{eq.defbetah}
\end{equation}%
Finally, we define a sequence $(J_{h})_{h\geq 0}$ of non-negative numbers as
follows: 
\begin{equation}
J_{h}:=\int_{A_{k_{h},\rho _{h}}}(u-k_{h})^{p^{\ast }}\,dx.  \label{eq.defJh}
\end{equation}%
%
%
%
%
%
%
%
%
%
%
%
%
%
%
%
%
%
%
%
%
%
%
%
%
%
%
%
%
%
%
%
%
%
%
%
%
%
%
%
%
%Introduce the following notation: 
%\begin{equation}
%\varphi :=\max \left\{ p^{\ast },\frac{2p}{q-p+2},\frac{p(\beta +1)}{p-1}%
%,\delta +1\right\}   \label{e:defphi}
%\end{equation}%
%and recall the definition of $\sigma $ in (\ref{e:defsigma}). The constant $%
%\sigma $ is positive due to the assumptions on the exponents, see also (\ref%
%{e:tau}). 
Then the following result holds.

\begin{proposition}
\label{prop.Jhestimate} For every real number $d\geq 1$, 
\begin{align}
J_{h+1}\leq & \ c_{\ast }\left( \tfrac{1}{R}\right) ^{\frac{p^{\ast }}{p-q+1}%
}\left( 1+\Vert u\Vert _{L^{p^{\ast }}(B_{R})}^{p^{\ast }}\right) ^{\frac{%
p^{\ast }}{p}\max \big\{\frac{1}{s_{0}},\frac{1}{s_{1}}\}}\times  \notag \\
& \times \tfrac{1}{d^{\frac{p^{\ast }}{p}\sigma }}\left( 2^{\frac{p^{\ast }}{%
p}p^{\ast }}\right) ^{h}J_{h}^{\frac{p^{\ast }}{p}\big(1-\max \big\{\frac{1}{%
s_{0}},\frac{1}{s_{1}}\}\big)},  \label{e:Jhp1leJh}
\end{align}%
where $\sigma $ is defined in (\ref{e:defsigma}); i.e., $\sigma :=p^{\ast
}-\max \left\{ \tfrac{p}{p-q+1};\;\tfrac{p^{\ast }}{s_{1}}+1;\;\tfrac{%
p^{\ast }}{s_{0}}+1;\;\theta ;\;\delta \right\} $ and $c_{\ast }$ positive
constant depending on the data, the $L^{s_{0}}$ and the $L^{s_{1}}$ norms in 
$B_{R_{0}}$ of $b_{0},b_{1}$, respectively, but it is independent of $u$ and 
$d$.
\end{proposition}

We notice that, by assumptions (\ref{e:hppstar}) and (\ref{e:hpgamma}), $%
\frac{p^{\ast }}{p}\big(1-\max\big\{\frac{1}{s_{0}},\frac{1}{s_{1}}\}\big)>1$%
. \bigskip

\begin{proof}[Proof of Proposition \protect\ref{prop.Jhestimate}]
We explicitly observe that, since $(\rho_h)_h$ is decreasing and $(k_h)_h$
is increasing, the sequence $(J_h)_h$ is decreasing: in fact, we have 
\begin{align}
J_{h+1} & = \int_{A_{k_{h+1},\rho_{h+1}}}(u-k_{h+1})^{p^*}\,d x \leq
\int_{A_{k_{h+1},\rho_h}}(u-k_{h+1})^{p^*}\, d x  \notag \\[0.1cm]
& \leq \int_{A_{k_{h+1},\rho_h}}(u-k_{h})^{p^*}\, d x \leq
\int_{A_{k_{h},\rho_h}}(u-k_{h})^{p^*}\, d x = J_h.  \label{eq.Jhdecreasing}
\end{align}

%Moreover, on account of \eqref{eq.choiceR0less1}
%(and since $\rho_h\leq R\leq R_0$), we have
%\begin{equation} \label{eq.Jhleqone}
%	\begin{split}
%		J_h & = \int_{A_{k_h,\rho_h}}(u-k_h)^{p^*}\,\d x
%		\leq \int_{A_{k_h,\rho_h}}|u|^{p^*}\,\d x \\[0.1cm]
%		& \leq \int_{B(x_0,R_0)}|u|^{p^*}\,\d x \leq 1. 
%	\end{split}
%\end{equation}
Finally, by taking into account the definitions of $J_{h}$, $k_{h}$ and $%
\rho _{h}$, we have 
\begin{equation}
\begin{split}
J_{h}& =\int_{A_{k_{h},\rho _{h}}}(u-k_{h})^{p^{\ast }}\,dx\geq
\int_{A_{k_{h+1},\rho _{h}}}(u-k_{h})^{p^{\ast }}\,dx \\
& \geq (k_{h+1}-k_{h})^{p^{\ast }}\,\big|A_{k_{h+1},\rho _{h}}\big|=\bigg(%
\frac{d}{2^{h+2}}\bigg)^{p^{\ast }}\,\big|A_{k_{h+1},\rho _{h}}\big|.
\end{split}
\label{eq.estimJhgeq}
\end{equation}%
Let $(\eta _{h})_{h\geq 0}$ be a sequence in $C_{c}^{\infty }(\mathbb{R})$
such that \textit{(i)}\thinspace \thinspace $0\leq \eta _{h}\leq 1$ on $%
\mathbb{R}^{n}$; \textit{(ii)}\thinspace \thinspace $\mathrm{supp}\,\eta
_{h}\subseteq B(x_{0},\overline{\rho }_{h})$ and $\eta _{h}\equiv 1$ on $%
B(x_{0},\rho _{h+1})$; \textit{(iii)}\thinspace \thinspace $|D\eta _{h}|\leq %
\displaystyle\tfrac{2^{h+4}}{R}$. In particular, $\eta _{h}\equiv 1\ $on $%
A_{k_{h+1},\rho _{h+1}}$, so we have 
\begin{align}
J_{h+1}^{\frac{p}{p^{\ast }}}& =\left( \int_{A_{k_{h+1},\rho
_{h+1}}}(u-k_{h+1})^{p^{\ast }}\,dx\right) ^{\frac{p}{p^{\ast }}}  \notag \\
& \leq \left( \int_{B(x_{0},\overline{\rho }_{h})}(\eta
_{h}(u-k_{h+1})_{+})^{p^{\ast }}\,dx\right) ^{\frac{p}{p^{\ast }}}  \notag \\
& \leq C_{S}^{p}\int_{B(x_{0},\overline{\rho }_{h})}\big|D(\eta
_{h}(u-k_{h+1})_{+})\big|^{p}\,dx,  \label{e:Jh+1CS}
\end{align}%
where $c_{S}$ is the Sobolev constant. We estimate the last integral {%
\begin{align}
& \int_{B(x_{0},\overline{\rho }_{h})}\big|D(\eta _{h}(u-k_{h+1})_{+})\big|%
^{p}\,dx  \notag \\
\leq & \int_{B(x_{0},\overline{\rho }_{h})}\Big\{|D\eta
_{h}|\,(u-k_{h+1})_{+}\,\,+\eta _{h}\,|Du|\,\chi _{A_{k_{h+1},\overline{\rho 
}_{h}}}\Big\}^{p}\,dx  \notag \\
\leq & c\,\bigg\{\bigg(\tfrac{2^{h+4}}{R}\bigg)^{p}\int_{A_{k_{h+1},%
\overline{\rho }_{h}}}(u-k_{h+1})^{p}\,dx+\int_{A_{k_{h+1},\overline{\rho }%
_{h}}}|Du|^{p}\,dx\bigg\}.  \label{e:etahbarrho}
\end{align}%
}Collecting (\ref{e:Jh+1CS}) and (\ref{e:etahbarrho}) we get 
\begin{equation}
J_{h+1}^{\frac{p}{p^{\ast }}}\leq c\bigg\{\bigg(\tfrac{2^{h+4}}{R}\bigg)%
^{p}\int_{A_{k_{h+1},\overline{\rho }_{h}}}(u-k_{h+1})^{p}\,dx+%
\int_{A_{k_{h+1},\overline{\rho }_{h}}}|Du|^{p}\,dx\bigg\}.
\label{e:etahbarrho2}
\end{equation}%
To estimate the last integral in (\ref{e:etahbarrho2}), we use the
Caccioppoli estimate (\ref{e:fineIIIstep}) with $k=k_{h+1}$, $\rho =%
\overline{\rho }_{h}$, $R=\rho _{h}$, 
%and note that, by (\ref{eq.Jhdecreasing}) 
%\begin{equation*}
%\Vert (u-k_{h+1})_{+}\Vert _{L^{p^{\ast }}(B_{\rho _{h}})}\leq J_{h}^{\frac{1%
%}{p^{\ast }}},
%\end{equation*}%
so obtaining 
\begin{align}
& \int_{A_{k_{h+1},\overline{\rho }_{h}}}|Du|^{p}\,dx\leq c\left( \tfrac{%
2^{h+3}}{R}\right) ^{\frac{p}{p-q+1}}\int_{A_{k_{h+1},\rho
_{h}}}(u-k_{h+1})^{\frac{p}{p-q+1}}\,dx  \notag \\
& +\,c\tfrac{2^{h+3}}{R}\Vert b_{1}\Vert _{L^{s_{1}}(B_{R_{0}})}J_{h}^{\frac{%
1}{p^{\ast }}}|A_{k_{h+1},\rho _{h}}|^{1-\frac{1}{s_{1}}-\frac{1}{p^{\ast }}}
\notag \\
& +\,c\,\Vert b_{0}+1\Vert _{L^{s_{0}}(B_{R_{0}})}J_{h}^{\frac{1}{p^{\ast }}%
}|A_{k_{h+1},\rho _{h}}|^{1-\frac{1}{s_{0}}-\frac{1}{p^{\ast }}}  \notag \\
\ & +c\,J_{h}^{\frac{2p}{p^{\ast }(q-p+2)}}|A_{k_{h+1},\rho _{h}}|^{1-\frac{%
2p}{p^{\ast }(q-p+2)}}+c\,J_{h}^{\frac{\theta }{p^{\ast }}}|A_{k_{h+1},\rho
_{h}}|^{1-\frac{\theta }{p^{\ast }}}  \notag \\
& +c\,J_{h}^{\frac{\delta }{p^{\ast }}}|A_{k_{h+1},\rho _{h}}|^{1-\frac{%
\delta }{p^{\ast }}}+c\,(1+k_{h+1}^{\tau })|A_{k_{h+1},\rho _{h}}|  \notag \\
& +c\,\Vert b_{1}\Vert _{L^{s_{1}}(B_{R_{0}})}|A_{k_{h+1},\rho _{h}}|^{1-%
\frac{1}{s_{1}}}.  \label{e:usoIIIstep}
\end{align}%
By (\ref{eq.estimJhgeq}) we get 
\begin{equation}
\big|A_{k_{h+1},\rho _{h+1}}\big|\leq \big|A_{k_{h+1},\rho _{h}}\big|\leq %
\bigg(\tfrac{2^{h+2}}{d}\bigg)^{p^{\ast }}J_{h}=4^{p^{\ast }}\,\bigg(\tfrac{%
2^{h}}{d}\bigg)^{p^{\ast }}J_{h}.  \label{abetah+1}
\end{equation}%
Collecting (\ref{e:usoIIIstep}), (\ref{abetah+1}), and using that $%
k_{h+1}\leq d$, we get 
\begin{align}
& \int_{A_{k_{h+1},\overline{\rho }_{h}}}|Du|^{p}\,dx\leq c\bigg(\tfrac{%
2^{h+3}}{R}\bigg)^{\frac{p}{p-q+1}}\int_{A_{k_{h+1},\rho _{h}}}(u-k_{h+1})^{%
\frac{p}{p-q+1}}\,dx  \notag \\
& +c\tfrac{2^{h}}{R}\Vert b_{1}\Vert _{L^{s_{1}}(B_{R_{0}})}\left( 4\,\tfrac{%
2^{h}}{d}\right) ^{p^{\ast }-\frac{p^{\ast }}{s_{1}}-1}J_{h}^{1-\frac{1}{%
s_{1}}}  \notag \\
& +c\,\Vert b_{0}+1\Vert _{L^{s_{0}}(B_{R_{0}})}\left( 4\,\tfrac{2^{h}}{d}%
\right) ^{p^{\ast }-\frac{p^{\ast }}{s_{0}}-1}J_{h}^{1-\frac{1}{s_{0}}}
\label{e:postusoIIIstep} \\
& +c\,\left\{ \bigg(4\tfrac{2^{h}}{d}\bigg)^{p^{\ast }-\frac{2p}{q-p+2}}+%
\bigg(4\tfrac{2^{h}}{d}\bigg)^{p^{\ast }-\theta }+\bigg(4\tfrac{2^{h}}{d}%
\bigg)^{p^{\ast }-\delta }\right\} J_{h}  \notag \\
& +c\,(1+d^{\tau })\,\bigg(4\tfrac{2^{h}}{d}\bigg)^{p^{\ast }}J_{h}+c\,\Vert
b_{1}\Vert _{L^{s_{1}}(B_{R_{0}})}\left( 4\,\tfrac{2^{h}}{d}\right)
^{p^{\ast }-\frac{p^{\ast }}{s_{1}}}J_{h}^{1-\frac{1}{s_{1}}}.  \notag
\end{align}%
We now estimate the integral at the right hand side. 
\begin{align*}
& \int_{A_{k_{h+1},\rho _{h}}}(u-k_{h+1})^{\frac{p}{p-q+1}}\,dx \\
\leq & \left( \int_{A_{k_{h+1},\rho _{h}}}(u-k_{h+1})^{p^{\ast }}\,dx\right)
^{\frac{p}{p^{\ast }(p-q+1)}}|A_{k_{h+1},\rho _{h}}|^{1-\frac{p}{p^{\ast
}(p-q+1)}}
\end{align*}%
that, by (\ref{abetah+1}), implies 
\begin{equation}
\int_{A_{k_{h+1},\rho _{h}}}(u-k_{h+1})^{\frac{p}{p-q+1}}\,dx\leq c\,\bigg(%
\tfrac{2^{h}}{d}\bigg)^{p^{\ast }\big(1-\frac{p}{p^{\ast }(p-q+1)}\big)%
}J_{h}.  \label{e:stimaIintegrale}
\end{equation}%
Collecting (\ref{e:postusoIIIstep}),(\ref{e:stimaIintegrale}) and using $%
k_{h+1}\leq d$, we get 
\begin{align}
& \int_{A_{k_{h+1},\overline{\rho }_{h}}}|Du|^{p}\,dx\leq c\,\bigg(\tfrac{1}{%
R}\bigg)^{\frac{p}{p-q+1}}\bigg(\tfrac{1}{d}\bigg)^{p^{\ast }\big(1-\frac{p}{%
p^{\ast }(p-q+1)}\big)}2^{p^{\ast }h}J_{h}  \notag \\
& +c\,\Vert b_{1}\Vert _{L^{s_{1}}(B_{R_{0}})}\tfrac{1}{R}\bigg(\tfrac{1}{d}%
\bigg)^{p^{\ast }-\frac{p^{\ast }}{s_{1}}-1}2^{p^{\ast }\big(1-\frac{1}{s_{1}%
}\big)h}J_{h}^{1-\frac{1}{s_{1}}}  \notag \\
& +c\,\Vert b_{0}+1\Vert _{L^{s_{0}}(B_{R_{0}})}\bigg(\tfrac{2^{h}}{d}\bigg)%
^{p^{\ast }-\frac{p^{\ast }}{s_{0}}-1}J_{h}^{1-\frac{1}{s_{0}}}  \notag \\
& +c\,\left\{ \bigg(\tfrac{2^{h}}{d}\bigg)^{p^{\ast }-\frac{2p}{q-p+2}}+%
\bigg(\tfrac{2^{h}}{d}\bigg)^{p^{\ast }-\theta }+\bigg(\tfrac{2^{h}}{d}\bigg)%
^{p^{\ast }-\delta }+(1+d^{\tau })\,\bigg(\tfrac{2^{h}}{d}\bigg)^{p^{\ast
}}\right\} J_{h}  \notag \\
& +c\Vert b_{1}\Vert _{L^{s_{1}}(B_{R_{0}})}\bigg(\tfrac{2^{h}}{d}\bigg)%
^{p^{\ast }-\frac{p^{\ast }}{s_{1}}}J_{h}^{1-\frac{1}{s_{1}}}
\label{e:postusoIIIstep2}
\end{align}%
with a constant $c$ depending on $n,p,q,\theta ,\delta ,s_{0},s_{1}$ and the
embedding Sobolev constant $c_{S}$, but depending neither on $d$, $h$ nor $u$%
. We now put together (\ref{e:etahbarrho2}) and (\ref{e:postusoIIIstep2});
taking into account that H\"{o}lder inequality and (\ref{abetah+1}) imply 
\begin{equation*}
\int_{A_{k_{h+1},\overline{\rho }_{h}}}(u-k_{h+1})^{p}\,dx\leq c\,\bigg(%
\tfrac{2^{h}}{d}\bigg)^{p^{\ast }-p}J_{h},
\end{equation*}%
we obtain 
\begin{align*}
J_{h+1}^{\frac{p}{p^{\ast }}}\leq & \,c\,\left( \bigg(\tfrac{1}{R}\bigg)^{p}%
\bigg(\tfrac{1}{d}\bigg)^{p^{\ast }\big(1-\frac{p}{p^{\ast }}\big)}+\bigg(%
\tfrac{1}{R}\bigg)^{\frac{p}{p-q+1}}\bigg(\tfrac{1}{d}\bigg)^{p^{\ast }\big(%
1-\frac{p}{p^{\ast }(p-q+1)}\big)}\right) 2^{p^{\ast }h}J_{h} \\
& +c\,\Vert b_{1}\Vert _{L^{s_{1}}(B_{R_{0}})}\tfrac{1}{R}\bigg(\tfrac{1}{d}%
\bigg)^{p^{\ast }-\frac{p^{\ast }}{s_{1}}-1}2^{p^{\ast }\big(1-\frac{1}{s_{1}%
}\big)h}J_{h}^{1-\frac{1}{s_{1}}} \\
& +c\,\Vert b_{0}+1\Vert _{L^{s_{0}}(B_{R_{0}})}\bigg(\tfrac{2^{h}}{d}\bigg)%
^{p^{\ast }-\frac{p^{\ast }}{s_{0}}-1}J_{h}^{1-\frac{1}{s_{0}}} \\
& +c\,\left\{ \bigg(\tfrac{2^{h}}{d}\bigg)^{p^{\ast }-\frac{2p}{q-p+2}}+%
\bigg(\tfrac{2^{h}}{d}\bigg)^{p^{\ast }-\theta }+\bigg(\tfrac{2^{h}}{d}\bigg)%
^{p^{\ast }-\delta }+(1+d^{\tau })\,\bigg(\tfrac{2^{h}}{d}\bigg)^{p^{\ast
}}\right\} J_{h} \\
& +\Vert b_{1}\Vert _{L^{s_{1}}(B_{R_{0}})}\bigg(\tfrac{2^{h}}{d}\bigg)%
^{p^{\ast }\big(1-\frac{1}{s_{1}}\big)}J_{h}^{1-\frac{1}{s_{1}}}.
\end{align*}%
%
%
%
%
%
%
%
%
%
%
%
%
%
%
%
%
%
%
%
%
%
%
%
%
%
%
%
%
%
%
%
%
%
%
%
%
%
% \begin{align*}\nonumber 
% 	J_{h+1}^{\frac{p}{p^*}} \le &  	\, c\,\left\{\bigg(\frac{2^{h}}{R}\bigg)^{p}	 	 \bigg(\bigg(\frac{2^{h}}{d}\bigg)^{p^*} \bigg)^{1-\frac{p}{p^*}}+
% 		\bigg(\frac{1}{R}\bigg)^{\frac{p}{p-q+1}}
% 	\bigg(\frac{1}{d}\bigg)^{p^*\big(1-\frac{p}{p^*(p-q+1)}\big)}\big( 2^{p^*}\big)^{h}\right\} J_h 
% 	\\ \nonumber & + c \,
% \|b_1\|_{L^{\gamma\frac{p-1}{p}}(B_{R_0})}
% \frac{1}{R} \bigg(
% \frac{1}{d}\bigg)^{p^*\big(1-\frac{1}{s_1}-\frac{1}{p^*}\big)}\big(  2^{p^*\big(1-\frac{1}{s_1}\big)}\big)^{h}J_h^{1-\frac{1}{s_1}} 
% \\ \nonumber & + c \, 	
% \|\tilde{b}_0\|_{L^{ s_1}(B_{R_0})}
% \bigg(\frac{2^{h}}{d}\bigg)^{p^* \big(1-\frac{1}{ s_1}-\frac{1}{p^*}\big)}
% J_h^{1-\frac{1}{ s_1}}
% \\ \nonumber & + c \, \left\{\bigg(\frac{2^{h}}{d}\bigg)^{\frac{2p}{q-p+2}}+
% \bigg(\frac{2^{h}}{d}\bigg)^{\frac{p(\beta+1)}{p-1}}+
% \bigg(\frac{2^{h}}{d}\bigg)^{\delta+1}
% +d^{\tau}
% \,\bigg(\frac{2^{h}}{d}\bigg)^{p^*} \right\}J_h\\   & +\|b_1\|_{L^{s_1}(B_{R_0})}
% \bigg(\frac{2^{h}}{d}\bigg)^{p^*\big(1-\frac{1}{s_1}\big)} J_h^{1-\frac{1}{s_1}}.
% \end{align*}
Notice that $J_{h}\leq \Vert u\Vert _{L^{p^{\ast }}(B_{R})}^{p^{\ast }}$ for
every $h\in \mathbb{N}$, so that 
\begin{equation*}
\max \{J_{h},J_{h}^{1-\frac{1}{s_{0}}},J_{h}^{1-\frac{1}{s_{1}}}\}\leq
(1+\Vert u\Vert _{L^{p^{\ast }}(B_{R})}^{p^{\ast }})^{\max \big\{\frac{1}{%
s_{0}},\frac{1}{s_{1}}\big\}}J_{h}^{1-\max \big\{\frac{1}{s_{0}},\frac{1}{%
s_{1}}\big\}}.
\end{equation*}%
Therefore 
\begin{align}
& J_{h+1}^{\frac{p}{p^{\ast }}}\leq \,c\,(1+\Vert u\Vert _{L^{p^{\ast
}}(B_{R})}^{p^{\ast }})^{\max \big\{\frac{1}{s_{0}},\frac{1}{s_{1}}\big\}%
}\left( 1+\Vert b_{1}\Vert _{L^{s_{1}}(B_{R_{0}})}+\Vert b_{0}+1\Vert
_{L^{s_{0}}(B_{R_{0}})}\right) \times  \notag \\
& \times \left\{ \left( \bigg(\tfrac{1}{R}\bigg)^{p}\bigg(\tfrac{1}{d}\bigg)%
^{p^{\ast }-p}+\bigg(\tfrac{1}{R}\bigg)^{\frac{p}{p-q+1}}\bigg(\tfrac{1}{d}%
\bigg)^{p^{\ast }-\frac{p}{p-q+1}}\right) 2^{p^{\ast }h}\right.  \notag \\
& +\tfrac{1}{R}\bigg(\tfrac{1}{d}\bigg)^{p^{\ast }-\frac{p^{\ast }}{s_{1}}%
-1}2^{p^{\ast }\big(1-\frac{1}{s_{1}}\big)h}+\bigg(\tfrac{2^{h}}{d}\bigg)%
^{p^{\ast }-\frac{p^{\ast }}{s_{0}}-1}  \notag \\
& +\bigg(\tfrac{2^{h}}{d}\bigg)^{p^{\ast }-\frac{2p}{q-p+2}}+\bigg(\tfrac{%
2^{h}}{d}\bigg)^{p^{\ast }-\theta }+\bigg(\tfrac{2^{h}}{d}\bigg)^{p^{\ast
}-\delta }+(1+d^{\tau })\,\bigg(\tfrac{2^{h}}{d}\bigg)^{p^{\ast }}  \notag \\
& \left. +\bigg(\tfrac{2^{h}}{d}\bigg)^{p^{\ast }-\frac{p^{\ast }}{s_{1}}%
}\right\} J_{h}^{1-\max \big\{\frac{1}{s_{0}},\frac{1}{s_{1}}\big\}}.
\label{e:Jh+1leJh}
\end{align}%
We now majorize the right hand side. Since $R_{0}\leq 1$ and $q\geq 1$, then 
\begin{equation*}
\max \bigg\{1;\tfrac{1}{R};\left( \tfrac{1}{R}\right) ^{\frac{p}{p-q+1}}%
\bigg\}\leq \left( \tfrac{1}{R}\right) ^{\frac{p}{p-q+1}}\qquad \text{for
every $R\in ]0,R_{0}]$}.
\end{equation*}%
Note that 
\begin{equation*}
\max \left\{ p;\tfrac{p}{p-q+1};\tfrac{2p}{q-p+2}\right\} =\tfrac{p}{p-q+1}%
\,;\;\;\;\;\;\max \left\{ \tfrac{p^{\ast }}{s_{1}}+1;\tfrac{p^{\ast }}{s_{1}}%
\right\} =\tfrac{p^{\ast }}{s_{1}}+1\,.
\end{equation*}%
Taking into account that $d\geq 1$ and denoting $\sigma $ as in (\ref%
{e:defsigma}), i.e. 
\begin{equation*}
\sigma :=p^{\ast }-\max \left\{ \tfrac{p}{p-q+1};\;\tfrac{p^{\ast }}{s_{1}}%
+1;\;\tfrac{p^{\ast }}{s_{0}}+1;\;\theta ;\;\delta \right\} ,
\end{equation*}%
by inequality (\ref{e:Jh+1leJh}) we obtain 
\begin{align*}
J_{h+1}^{\frac{p}{p^{\ast }}}\leq & c_{0}\left( \tfrac{1}{R}\right) ^{\frac{p%
}{p-q+1}}\left( 1+\Vert b_{1}\Vert _{L^{s_{1}}(B_{R_{0}})}+\Vert
b_{0}+1\Vert _{L^{s_{0}}(B_{R_{0}})}\right) \times \\
& \qquad \qquad \times \tfrac{1}{d^{\sigma }}\big(2^{p^{\ast }}\big)%
^{h}(1+\Vert u\Vert _{L^{p^{\ast }}(B_{R})}^{p^{\ast }})^{\max \big\{\frac{1%
}{s_{0}},\frac{1}{s_{1}}\big\}}J_{h}^{1-\max \big\{\frac{1}{s_{0}},\frac{1}{%
s_{1}}\big\}}.
\end{align*}%
Raising at the power $\frac{p^{\ast }}{p}$, we get (\ref{e:Jhp1leJh}) with 
\begin{equation}
c_{\ast }:=\max \{1,c_{0}^{\frac{p^{\ast }}{p}}\}\left( 1+\Vert b_{1}\Vert
_{L^{s_{1}}(B_{R_{0}})}+\Vert b_{0}+1\Vert _{L^{s_{1}}(B_{R_{0}})}\right) ^{%
\frac{p^{\ast }}{p}}.  \label{e:Jcstar}
\end{equation}
\end{proof}

\subsection{Proof of the local boundedness result}

With Proposition \ref{prop.Jhestimate} at hand, we are ready to provide the
proof of our main result, namely Theorem \ref{t:mainboundeness}. Before
doing this, we remind the following very classical lemma of Real Analysis
(see, e.g., \cite[Lemma 7.1]{Giusti 2003 book}).

\begin{lemma}
\label{lem.realanalysis} {Let $(z_{h})_{h\geq 0}$ be a sequence of positive
real numbers satisfying the fol\-low\-ing recursive relation 
\begin{equation}
z_{h+1}\leq L\,\zeta ^{h}z_{h}^{1+\alpha }\qquad (h\in \mathbb{N}\cup \{0\}),
\label{eq.estimzhlemma}
\end{equation}%
where $L,\alpha >0$ and $\zeta >1$. If $z_{0}\leq L^{-\frac{1}{\alpha }%
}\zeta ^{-\frac{1}{\alpha ^{2}}}$, then $z_{h}\leq \zeta ^{-\frac{h}{\alpha }%
}z_{0}\quad $for every $h\geq 0$. In particular, $z_{h}\rightarrow 0$ as $%
h\rightarrow \infty $.}
\end{lemma}

\begin{proof}[Proof of Theorem \protect\ref{t:mainboundeness}]
Let $u\in W_{\mathrm{loc}}^{1,q}(\Omega )$ be a weak solution to (\ref%
{elliptic equation}) under the assumptions in Section \ref{ss:ipotesi}.
Consider $B_{R_{0}}(x_{0})\Subset \Omega $ with $0<R_{0}\leq 1$. Moreover,
let $d\geq 1$ (to be chosen later on) and let $(J_{h})_{h\geq 0}$ be the
sequence defined in (\ref{eq.defJh}). Owing to Proposition \ref%
{prop.Jhestimate}, for every $R\leq R_{0}$ we have the estimate 
\begin{equation}
J_{h+1}\leq L\,(2^{\frac{p^{\ast }}{p}p^{\ast }})^{h}\,J_{h}^{1+\alpha
}\qquad (h\in \mathbb{N}\cup \{0\}),  \label{eq.estimJhproofmain}
\end{equation}%
where $\alpha $ is 
\begin{equation}
\alpha :=\frac{p^{\ast }}{p}\big(1-\max \big\{\frac{1}{s_{0}},\frac{1}{s_{1}}%
\big\}\big)-1,  \label{eq.alpha}
\end{equation}%
and the constant $L$ is given by 
\begin{equation*}
L:=c_{\ast }\,(1+\Vert u\Vert _{L^{p^{\ast }}(B_{R})}^{p^{\ast }})^{\frac{%
p^{\ast }}{p}\max \big\{\frac{1}{s_{0}},\frac{1}{s_{1}}\big\}}\,\tfrac{1}{d^{%
\frac{p^{\ast }}{p}\sigma }}\left( \tfrac{1}{R}\right) ^{\frac{p^{\ast }}{%
p-q+1}},
\end{equation*}%
where $\sigma $ is defined in (\ref{e:defsigma}) and $c_{\ast }$,
independent of $d$, is defined as in (\ref{e:Jcstar}). 
%	\[C:= c_*\left(1+\|b_1\|_{L^{s_1}(B_{R_0})}+\|\tilde{b}_0\|_{L^{
%		s_1}(B_{R_0})}\right)^{\frac{p^{\ast }}{p}}  \] with $c_* > 0$ independent of $d$. 
We now claim that it is possible to choose $d\geq 1$ in such a way that 
\begin{equation}
J_{0}:=\int_{A_{\frac{d}{2},R}}\left( u-\tfrac{d}{2}\right) ^{p^{\ast
}}\,dx\leq L^{-1/\alpha }\,(2^{\frac{p^{\ast }}{p}p^{\ast }})^{-1/\alpha
^{2}}.  \label{eq.claimJzero}
\end{equation}%
%
%
%
%
%
%
%
%
%
%
%
%
%
%
%
%
%
%
%
%
%
%
%
%
%
%
%
%
%
%
%
%
%
%
%
%
%
%
%
%
% that is, by \eqref{eq.defJh}, 
%\begin{align}\nonumber
%		& J_0 =  \int_{A_{\frac{d}{2},R}}(u-\frac{d}{2})^{p^*}\,d x 
%		\\[0.15cm]
%		& \qquad\quad
%		\leq \bigg\{{C}
%		\,\big(1+\|u\|_{L^{\nu^*}(B(x_0,R_0))}^{(q-p)\sigma'}\big)\cdot
%		\frac{1}{\beta^{p\sigma'(1-\frac{q\sigma'}{\nu^*})}}
%		\bigg(\frac{1}{R_0}\bigg)^{p\sigma'}\bigg\}^{-1/\alpha}\,(2^{\frac{p^{\ast }}{p}%
%			\varphi })^{-1/\alpha^2}. \label{eq.claimJzero} 
%\end{align}
In fact, by definition of $J_{0}$ and since $u\in L_{\loc}^{p^{\ast
}}(\Omega )$, we have 
\begin{equation*}
J_{0}\leq \int_{B_{R}}|u|^{p^{\ast }}\,dx<\infty ;
\end{equation*}%
thus, reminding that $R_{0}\leq 1$, condition (\ref{eq.claimJzero}) is
clearly fulfilled if we choose 
\begin{equation*}
d:=\tfrac{(c_{\ast })^{\frac{p}{p^{\ast }\sigma }}}{R^{\frac{p}{\sigma
(p-q+1)}}}\,2^{\frac{p^{\ast }}{\alpha \sigma }}\,\left( 1+\Vert u\Vert
_{L^{p^{\ast }}(B(x_{0},R))}^{p^{\ast }}\right) ^{\frac{p\alpha }{p^{\ast
}\sigma }+\max \big\{\frac{1}{s_{0}},\frac{1}{s_{1}}\big\}\frac{1}{\sigma }},
\end{equation*}%
that is, taking into account (\ref{eq.alpha}), 
\begin{equation}
d:=\tfrac{(c_{\ast })^{\frac{p}{p^{\ast }\sigma }}}{R^{\frac{p}{\sigma
(p-q+1)}}}\,2^{\frac{p^{\ast }}{\alpha \sigma }}\,\left( 1+\Vert u\Vert
_{L^{p^{\ast }}(B(x_{0},R))}^{p^{\ast }}\right) ^{\frac{p^{\ast }-p}{p^{\ast
}\sigma }}.  \label{eq.choicebeta}
\end{equation}%
Notice that $d\geq 1$, because $c_{\ast }\geq 1$ and $R_{0}\leq 1$. With (%
\ref{eq.claimJzero}) at hand and $d$ as in (\ref{eq.choicebeta}), we are
entitled to apply Lemma \ref{lem.realanalysis}. As a consequence, we obtain 
\begin{equation}
\lim_{h\rightarrow \infty }J_{h}=\lim_{h\rightarrow \infty
}\int_{A_{k_{h},\rho _{h}}}(u-k_{h})^{p^{\ast
}}\,dx=\int_{A_{d,R/2}}(u-d)^{p^{\ast }}\,dx=0.  \label{eq.limitJhzero}
\end{equation}%
Since, by definition, $u-d>0$ on $A_{d,R/2}$, from (\ref{eq.limitJhzero}) we
then conclude that 
\begin{equation*}
\big|A_{d,R/2}\big|=0,\qquad \text{whence}\quad \text{$u\leq d$ for
a.e.\thinspace $x\in B_{R/2}(x_{0})$}.
\end{equation*}%
To prove that $u$ is locally bounded from below, we can reason analogously,
using the sub-level sets of $u$. 
%To prove that $u$ is also bounded from below by $-d$ it suffices to observe that,
%setting $v := -u$, then $v\in W^{1,1}_{\loc}(\Omega)$ is a solution  to....
So we obtain that there exists $c_{\ast }^{\prime }$ such that $-u\leq
d^{\prime }$ a.e. in $B_{\frac{R}{2}}$, with 
\begin{equation*}
d^{\prime }=\tfrac{(c_{\ast }^{\prime })^{\frac{p}{p^{\ast }\sigma }}}{R^{%
\frac{p}{\sigma (p-q+1)}}}\,2^{\frac{p^{\ast }}{\alpha \sigma }}\,\left(
1+\Vert u\Vert _{L^{p^{\ast }}(B(x_{0},R))}^{p^{\ast }}\right) ^{\frac{%
p^{\ast }-p}{p^{\ast }\sigma }}.
\end{equation*}%
We have so proved that $u\in L^{\infty }(B_{R/2}(x_{0}))$. Due to the
arbitrariness of $x_{0}$ and $R_{0}$, we get $u\in L_{\mathrm{loc}}^{\infty
}(\Omega )$ and (\ref{e:stimalim}) follows.
\end{proof}

%
%
%
%\[\lambda=2^{\frac{p^*}{p}\varphi}\]
%with 
%%\[\varphi:=\max\{p^*,\frac{2p}{q-p+2},
%%\frac{p(\beta+1)}{p-1}, \delta+1, p^*-\tau,p^*-\frac{p^*}{\gamma} 
%%\}\]
%%that is 
%\[\varphi:=\max\{p^*,\frac{2p}{q-p+2},
%\frac{p(\beta+1)}{p-1}, \delta+1\}\]
%
%and, using $d\ge 1$,  \begin{align}\nonumber\sigma:=\frac{p^*}{p}\cdot \min\bigg\{&
%	p^* \bigg(1-\frac{p}{p^*(p-q+1)}\bigg), \frac{2p}{q-p+2}, \frac{p(\beta+1)}{p-1}, \\ &\delta+1, p^*-\tau, p^* (1-\frac{1}{\gamma})
%	\bigg\}\label{e:sigma}\end{align}

\section{Local Lipschitz continuity and higher differentiability\label%
{Section: Local Lipschitz continuity}}

We start with the assumptions of Section \ref{Section: assumptions for the
local Lipschitz continuity} in order to prove Theorem \ref{final theorem}.
The other local Lipschitz continuity Theorem \ref{final theorem 2} will need
some specific considerations which we will propose below. Thus again, as in
Definition \ref{aweaksolDEF}, we consider a weak solution $u\in W_{\mathrm{%
loc}}^{1,q}\left( \Omega \right) $\ of the differential equation 
\begin{equation}
\sum_{i=1}^{n}\frac{\partial }{\partial x_{i}}a^{i}\left( x,u\left( x\right)
,Du\left( x\right) \right) =b\left( x,u\left( x\right) ,Du\left( x\right)
\right) ,\;\;\;\;\;x\in \Omega \,,  \label{elliptic equation in Section 6}
\end{equation}%
under the ellipticity condition 
\begin{equation}
\sum_{i,j=1}^{n}\frac{\partial a^{i}}{\partial \xi _{j}}\lambda _{i}\lambda
_{j}\geq m\left( 1+\left\vert \xi \right\vert ^{2}\right) ^{\frac{p-2}{2}%
}\left\vert \lambda \right\vert ^{2}\;.  \label{ellipticity in Section 6}
\end{equation}%
and the order-one growth conditions (\ref{growth 0-a}), (\ref{growth 0-b}).
By Lemma \ref{Lemma 1} the assumptions (\ref{(i) in the Section}) and (\ref%
{(ii) in the Section}) hold. Therefore by Theorem \ref{t:mainboundeness} the
weak solution $u\in W_{\mathrm{loc}}^{1,q}\left( \Omega \right) $\ is
locally bounded in $\Omega $.

We read again assumptions (\ref{growth 0-a}), (\ref{growth 0-b}) by taking
into account the local boundedness of $u$: for every open set $\Omega
^{\prime }$, whose closure is contained in $\Omega ,$ there exists a
constant $L>0$ such that $\left\Vert u\right\Vert _{L^{\infty }\left( \Omega
^{\prime }\right) }\leq L$; therefore, by (\ref{growth 0-a}), (\ref{growth
0-b}), there exist a positive constant $M\left( L\right) $ (depending on $%
\Omega ^{\prime }$ and $L$; precisely $M\left( L\right) =\max \left\{
M;L^{\alpha };L^{\beta -1}\right\} $) such that, for every $x\in \Omega
^{\prime }$, $\lambda ,\xi \in \mathbb{R}^{n}$ and for$\;\left\vert
u\right\vert \leq L$ 
\begin{equation}
\left\vert \frac{\partial a^{i}}{\partial \xi _{j}}\right\vert \leq M\left(
L\right) \left( 1+\left\vert \xi \right\vert ^{2}\right) ^{\frac{q-2}{2}%
}\,,\;\;\;\;\;\left\vert \frac{\partial a^{i}}{\partial u}\right\vert \leq
M\left( L\right) \left( 1+\left\vert \xi \right\vert ^{2}\right) ^{\frac{%
p+q-4}{4}}\,.  \label{growth 0 in Section 6}
\end{equation}%
Conditions (\ref{growth 0 in Section 6}) correspond to the assumptions (3.7)
in Marcellini \cite{Marcellini 2023}. Together with the ellipticity
condition (\ref{elliptic equation in Section 6}) and the growth conditions (%
\ref{growth 1}), (\ref{growth 2}) they give the full landscape in order to
state that the vector field $a\left( x,u,\xi \right) =\left( a^{i}\left(
x,u,\xi \right) \right) _{i=1,\ldots ,n}$ satisfies all the assumptions
taken under consideration in \cite{Marcellini 2023} to study the local
Lipschitz continuity of the weak solution $u$. To this aim, it remains only
to analyze the assumption on the right hand side $b\left( x,u,\xi \right) $
in the differential equation (\ref{elliptic equation in Section 6}).

In this paper the growth assumption on $b\left( x,u,\xi \right) $ in (\ref{b
locally bounded}) is (here with abuse of notation $M\left( L\right) =\max
\left\{ M;ML\right\} $) 
\begin{equation}
\left\vert b\left( x,u,\xi \right) \right\vert \leq M\left( L\right) \left(
1+\left\vert \xi \right\vert ^{2}\right) ^{\frac{p+q-2}{4}}+b_{0}\left(
x\right) \,,  \label{growth for b in Section 6}
\end{equation}%
a.e.$\;x\in \Omega ^{\prime }$, and for all $\lambda ,\xi \in \mathbb{R}^{n}$
and$\;\left\vert u\right\vert \leq L$. The difference with the growth
assumption in \cite{Marcellini 2023} is the addendum $b_{0}\left( x\right) $
in the right hand side of (\ref{growth for b in Section 6}), with $b_{0}\in
L_{\mathrm{loc}}^{s_{0}}\left( \Omega \right) $ for some $s_{0}>n$, which is
posed equal to zero in \cite{Marcellini 2023}. Therefore in the following we
analyze in which way it is possible to modify the argument of \cite%
{Marcellini 2023} in order to handle this term.

\subsection{Proof of Theorems \protect\ref{final theorem} and \protect\ref%
{final theorem 2}}

As already said, we start with the proof of Theorem \ref{final theorem}.
Under the notation of Section 5 in \cite{Marcellini 2023} we consider the
equation (\ref{elliptic equation in Section 6}) and a test function $\varphi
\in W_{0}^{1,q}\left( \Omega ^{\prime }\right) $ of the form 
\begin{equation*}
\varphi =\Delta _{-h}\left( \eta ^{2}g\left( \Delta _{h}u\right) \right) ,
\end{equation*}%
where $\Delta _{-h}\psi $ denote as usual the difference quotient of a
function $\psi (x)$ in the $s-$direction and $g$ is a Lipschitz continuous
function $g:\mathbb{R}\rightarrow \mathbb{R}$, with $0<g^{\prime }\left(
t\right) \leq L$ for all $t\in \mathbb{R}$, and $\eta \in C_{0}^{1}\left(
\Omega ^{\prime }\right) $, $\eta \geq 0$, $\Omega ^{\prime }\subset \subset
\Omega $. We can see that $\varphi \in W_{0}^{1,q}\left( \Omega ^{\prime
}\right) $. Precisely, following Section 5.5.2 in \cite{Marcellini 2023},
for generic $k\in \mathbb{N}$ and $\beta \geq 0$ we choose $g\left( t\right)
=t\left( 1+t^{2}\right) ^{\beta /2}$ when $t\in \left[ -k,k\right] $ and $%
g\left( t\right) $ affine out of the interval $\left[ -k,k\right] $ in such
a way that globally $g\in C^{1}\left( \mathbb{R}\right) $. For reader's
convenience we adopt here the symbols in \cite{Marcellini 2023}; we note
that $\beta$ here is different and independent of the same symbol used in
Section \ref{Section: Local boundedness}. For $t\in \left[ -k,k\right] $ the
derivative of $g$ holds $g^{\prime }\left( t\right) =\left( 1+t^{2}\right)
^{\beta /2-1}\left( 1+\left( \beta +1\right) t^{2}\right) $; then 
\begin{equation}
0<g^{\prime }\left( t\right) \leq \left( \beta +1\right) \left(
1+t^{2}\right) ^{\beta /2}  \label{derivative of g(t)}
\end{equation}%
for all $t\in \left[ -k,k\right] $, and also for all $t\in \mathbb{R}$,
since $g^{\prime }\left( t\right) $ is constant out of the interval $\left[
-k,k\right] $; precisely $g^{\prime }\left( t\right) =g^{\prime }\left(
k\right) =g^{\prime }\left( -k\right) $ when $t\notin \left[ -k,k\right] $,
and thus for such $t-$values $g^{\prime }\left( t\right) =g^{\prime }\left(
k\right) \leq \left( \beta +1\right) \left( 1+k^{2}\right) ^{\beta /2}\leq
\left( \beta +1\right) \left( 1+t^{2}\right) ^{\beta /2}$.

We insert $\varphi =\Delta _{-h}\left( \eta ^{2}g\left( \Delta _{h}u\right)
\right) $ in the weak form of the differential equation (\ref{elliptic
equation in Section 6}) and we obtain 
\begin{equation}
\int_{\Omega }\sum_{i=1}^{n}a^{i}\left( x,u,Du\left( x\right) \right) \left(
\Delta _{-h}\left( \eta ^{2}g\left( \Delta _{h}u\right) \right) \right)
_{x_{i}}dx\,  \label{weak equation 1}
\end{equation}%
\begin{equation*}
+\,\int_{\Omega }b\left( x,u,Du\right) \left( \Delta _{-h}\left( \eta
^{2}g\left( \Delta _{h}u\right) \right) \right) \,dx=0.
\end{equation*}

The integral with the vector field $a\left( x,u,\xi \right) =\left(
a^{i}\left( x,u,\xi \right) \right) _{i=1,\ldots ,n}$ can be estimated
exactly as done in Section 5 of \cite{Marcellini 2023}. In fact, by
proceeding as in subsection 5.2 of \cite{Marcellini 2023}, we get 
\begin{equation}
\frac{1}{c}\int_{0}^{1}dt\int_{\Omega }\eta ^{2}g^{\prime }\left( \Delta
_{h}u\right) \left( 1+\left\vert \left( 1-t\right) Du\left( x\right)
+tDu\left( x+the_{s}\right) \right\vert ^{2}\right) ^{\frac{p-2}{2}%
}\left\vert \Delta _{h}Du\right\vert ^{2}\,dx\,  \label{stima-Paolo}
\end{equation}%
\begin{equation*}
\leq \int_{0}^{1}dt\int_{\Omega }\eta ^{2}g^{\prime }\left( \Delta
_{h}u\right) \left( 1+\left\vert Du\left( x\right) \right\vert
^{2}+\left\vert Du\left( x+the_{s}\right) \right\vert ^{2}\right) ^{\frac{q}{%
2}}\,dx\,
\end{equation*}%
\begin{equation*}
+\int_{0}^{1}dt\int_{\Omega }2\eta \left\vert D\eta \right\vert \cdot
\left\vert g\left( \Delta _{h}u\right) \right\vert \left( 1+\left\vert
Du\left( x\right) \right\vert ^{2}+\left\vert Du\left( x+the_{s}\right)
\right\vert ^{2}\right) ^{\frac{q-1}{2}}\,dx\,
\end{equation*}%
\begin{equation*}
+\int_{0}^{1}dt\int_{\Omega }\left\vert D\eta \right\vert ^{2}\cdot \frac{%
g^{2}\left( \Delta _{h}u\right) }{g^{\prime }\left( \Delta _{h}u\right) }%
\left( 1+\left\vert Du\left( x\right) \right\vert ^{2}+\left\vert Du\left(
x+the_{s}\right) \right\vert ^{2}\right) ^{\frac{q-2}{2}}\,dx\,
\end{equation*}%
\begin{equation*}
+\int_{\Omega }b(x,u,Du)\left( \Delta _{h}(\eta ^{2}g(\Delta _{h}u)\right)
\,dx\,.
\end{equation*}

\subsection{Estimate of the right hand side $b$}

For the integral related to the term with $b\left( x,u,\xi \right) $ we use
the growth assumption (\ref{growth for b in Section 6}): $\left\vert b\left(
x,u,\xi \right) \right\vert \leq M\left( L\right) \left( 1+\left\vert \xi
\right\vert ^{2}\right) ^{\frac{p+q-2}{4}}+b_{0}\left( x\right) $ and we get 
\begin{equation*}
\;\;\;\left\vert \int_{\Omega }b\left( x,u,Du\right) \left( \Delta
_{-h}\left( \eta ^{2}g\left( \Delta _{h}u\right) \right) \right)
\,dx\right\vert
\end{equation*}%
\begin{equation*}
\leq M\left( L\right) \left\vert \int_{\Omega }\left( 1+\left\vert
Du\right\vert ^{2}\right) ^{\frac{p+q-2}{4}}\cdot \left( \Delta _{-h}\left(
\eta ^{2}g\left( \Delta _{h}u\right) \right) \right) \,dx\right\vert
\end{equation*}%
\begin{equation*}
+\left\vert \int_{\Omega }\left\vert b_{0}\left( x\right) \right\vert \cdot
\left\vert \Delta _{-h}\left( \eta ^{2}g\left( \Delta _{h}u\right) \right)
\right\vert \,dx\right\vert \,.
\end{equation*}%
Reasoning as in \cite{Marcellini 2023} (see in particular Section 5.3.7 of 
\cite{Marcellini 2023}) and it can be estimated as in (5.15) of \cite%
{Marcellini 2023}, %
%
%The above term in (\ref{Section 6 - formula 2}) is identical to the
%correspond term in \cite{Marcellini 2023} (see in particular Section 5.3.7
%of \cite{Marcellini 2023}) and it can be estimated as in (5.15) of \cite%
%{Marcellini 2023}: 
%
we represent $\Delta _{-h}\left( \eta ^{2}g\left( \Delta _{h}u\right)
\right) $ in this way 
\begin{equation*}
\Delta _{-h}\left( \eta ^{2}g\left( \Delta _{h}u\right) \right)
=\int_{0}^{1}\left( 2\eta \eta _{x_{s}}\,g\left( \Delta _{h}u\right) +\eta
^{2}g^{\prime }\left( \Delta _{h}u\right) \Delta _{h}u_{x_{s}}\right) \,dt\,,
\end{equation*}%
where the arguments in the last integrands are $x-the_{s}$\thinspace .
Therefore 
\begin{equation}
\;\;\;\left\vert \int_{\Omega }b\left( x,u,Du\right) \left( \Delta
_{-h}\left( \eta ^{2}g\left( \Delta _{h}u\right) \right) \right)
\,dx\right\vert  \label{Section 6 - formula 1}
\end{equation}%
\begin{equation}
\leq M\left( L\right) \int_{\Omega }\left( 1+\left\vert Du\right\vert
^{2}\right) ^{\frac{p+q-2}{4}}\cdot \left( \int_{0}^{1}\left\vert 2\eta \eta
_{x_{s}}\,g\left( \Delta _{h}u\right) +\eta ^{2}g^{\prime }\left( \Delta
_{h}u\right) \Delta _{h}u_{x_{s}}\right\vert \,dt\,\right) \,dx
\label{Section 6 - formula 2}
\end{equation}%
\begin{equation}
+\int_{0}^{1}\left\vert \int_{\Omega }\left\vert b_{0}\left( x\right)
\right\vert \cdot \left\vert 2\eta \eta _{x_{s}}\,g\left( \Delta
_{h}u\right) +\eta ^{2}g^{\prime }\left( \Delta _{h}u\right) \Delta
_{h}u_{x_{s}}\right\vert \,dx\right\vert \,dt\,.
\label{Section 6 - formula 4}
\end{equation}%
The above term in (\ref{Section 6 - formula 2}) is identical to the
correspondent term in \cite{Marcellini 2023} (see in particular Section
5.3.7 of \cite{Marcellini 2023}) and it can be estimated as in (5.15) of 
\cite{Marcellini 2023}. Thus we limit here to estimate the addendum in (\ref%
{Section 6 - formula 4}). Recalling that $g^{\prime }\left( t\right) >0$ for
all $t\in \mathbb{R}$, we start by using the inequalities 
\begin{equation*}
2\eta \left\vert \eta _{x_{s}}\right\vert \,\left\vert b_{0}\left( x\right)
\right\vert \,\left\vert g\left( \Delta _{h}u\right) \right\vert =2\eta
\left\vert \eta _{x_{s}}\right\vert \,\left\vert b_{0}\left( x\right)
\right\vert \,\frac{\left\vert g\left( \Delta _{h}u\right) \right\vert }{%
\left( g^{\prime }\left( \Delta _{h}u\right) \right) ^{1/2}}\left( g^{\prime
}\left( \Delta _{h}u\right) \right) ^{1/2}
\end{equation*}%
\begin{equation*}
\leq \left\vert D\eta \right\vert ^{2}\frac{g^{2}\left( \Delta _{h}u\right) 
}{g^{\prime }\left( \Delta _{h}u\right) }+\eta ^{2}b_{0}^{2}\left( x\right)
\,g^{\prime }\left( \Delta _{h}u\right) \,;
\end{equation*}%
\begin{equation*}
\left\vert b_{0}\left( x\right) \right\vert \cdot \left\vert \Delta
_{h}u_{x_{s}}\right\vert \leq \varepsilon \left\vert \Delta
_{h}u_{x_{s}}\right\vert ^{2}+\frac{1}{4\varepsilon }b_{0}^{2}\left(
x\right) \,.
\end{equation*}%
As before we denote by $\Omega ^{\prime }={\supp}\eta $, which is a compact
set contained in $\Omega $. Moreover, since $b_{0}\in L_{\mathrm{loc}%
}^{s_{0}}\left( \Omega \right) $ for $s_{0}>n$, and in particular $b_{0}\in
L^{s_{0}}\left( \Omega ^{\prime }\right) $, we can also use H\"{o}lder
inequality with exponents $\frac{s_{0}}{2}$ and $\frac{s_{0}}{s_{0}-2}$.
From (\ref{Section 6 - formula 4}) we obtain 
\begin{equation}
\left\vert \int_{\Omega }\left\vert b_{0}\left( x\right) \right\vert \cdot
\left\vert \Delta _{-h}\left( \eta ^{2}g\left( \Delta _{h}u\right) \right)
\right\vert \,dx\right\vert  \label{Section 6 - formula 5}
\end{equation}%
\begin{equation*}
=\int_{0}^{1}\left\vert \int_{\Omega }\left\vert b_{0}\left( x\right)
\right\vert \cdot \left\vert 2\eta \eta _{x_{s}}\,g\left( \Delta
_{h}u\right) +\eta ^{2}g^{\prime }\left( \Delta _{h}u\right) \Delta
_{h}u_{x_{s}}\right\vert \,dx\right\vert \,dt\,
\end{equation*}%
\begin{equation*}
\leq \int_{0}^{1}\,dt\int_{\Omega }\left\vert D\eta \right\vert ^{2}\frac{%
g^{2}\left( \Delta _{h}u\right) }{g^{\prime }\left( \Delta _{h}u\right) }%
\,dx+\int_{0}^{1}\,dt\int_{\Omega }\eta ^{2}b_{0}^{2}\left( x\right)
\,g^{\prime }\left( \Delta _{h}u\right) \,dx
\end{equation*}%
\begin{equation*}
+\varepsilon \int_{\Omega }\eta ^{2}g^{\prime }\left( \Delta _{h}u\right)
\left\vert \Delta _{h}u_{x_{s}}\right\vert ^{2}\,dx+\frac{1}{4\varepsilon }%
\int_{0}^{1}\,dt\int_{\Omega }\eta ^{2}b_{0}^{2}\left( x\right) g^{\prime
}\left( \Delta _{h}u\right) \,dx
\end{equation*}%
\begin{equation*}
\leq \int_{0}^{1}\,dt\int_{\Omega }\left\vert D\eta \right\vert ^{2}\frac{%
g^{2}\left( \Delta _{h}u\right) }{g^{\prime }\left( \Delta _{h}u\right) }%
\,dx+\varepsilon \int_{\Omega }\eta ^{2}g^{\prime }\left( \Delta
_{h}u\right) \left\vert \Delta _{h}u_{x_{s}}\right\vert ^{2}\,dx
\end{equation*}%
\begin{equation*}
+\left( 1+\frac{1}{4\varepsilon }\right) \left( \int_{\Omega ^{\prime
}}\left\vert b_{0}\left( x\right) \right\vert ^{s_{0}}\,dx\right) ^{\frac{2}{%
s_{0}}}\cdot \int_{0}^{1}\,dt\left( \int_{\Omega }\left( \eta ^{2}g^{\prime
}\left( \Delta _{h}u\right) \right) ^{\frac{s_{0}}{s_{0}-2}}\,dx\right) ^{%
\frac{s_{0}-2}{s_{0}}}\,.
\end{equation*}%
With the aim to estimate the last addendum in (\ref{Section 6 - formula 5})
we observe that the exponent $\frac{s_{0}}{s_{0}-2}>1$ is strictly less than 
$\frac{n}{n-2}$, since $s_{0}>n$ (for simplicity, we limit here to consider
the details for the case $n>2;$ for $n=2$ we can proceed similarly with
small modifications)). We represent $\frac{s_{0}}{s_{0}-2}$ as \textit{%
convex combination} of $1$ and $\frac{n}{n-2}$ 
\begin{equation*}
\frac{s_{0}}{s_{0}-2}=t+\frac{n}{n-2}\left( 1-t\right) ,\;\;\;\;\;\;\text{%
with}\;\;\;t=\frac{s_{0}-n}{s_{0}-2}\;\;\;\;\text{and}\;\;\;1-t=\frac{n-2}{%
s_{0}-2}\,.
\end{equation*}%
Let $\lambda $ be a positive real parameter that we will fix later; a
computation shows that $\lambda ^{-\frac{n}{s_{0}-n}t+\frac{n}{n-2}\left(
1-t\right) }=1$. By H\"{o}lder inequality with exponents $\frac{1}{t}$ and $%
\frac{1}{1-t}$ 
\begin{equation*}
\int_{\Omega }\left( \eta ^{2}g^{\prime }\left( \Delta _{h}u\right) \right)
^{\frac{s_{0}}{s_{0}-2}}\,dx=\int_{\Omega }\left( \eta ^{2}g^{\prime }\left(
\Delta _{h}u\right) \right) ^{t+\frac{n}{n-2}\left( 1-t\right) }\,dx
\end{equation*}%
\begin{equation*}
=\int_{\Omega }\left( \lambda ^{-\frac{n}{s_{0}-n}}\eta ^{2}g^{\prime
}\left( \Delta _{h}u\right) \right) ^{t}\left( \lambda \eta ^{2}g^{\prime
}\left( \Delta _{h}u\right) \right) ^{\frac{n}{n-2}\left( 1-t\right) }\,dx
\end{equation*}%
\begin{equation*}
\leq \left( \lambda ^{-\frac{n}{s_{0}-n}}\int_{\Omega }\eta ^{2}g^{\prime
}\left( \Delta _{h}u\right) \,dx\right) ^{t}\left( \lambda ^{\frac{n}{n-2}%
}\int_{\Omega }\left( \eta ^{2}g^{\prime }\left( \Delta _{h}u\right) \right)
^{\frac{n}{n-2}}\,dx\right) ^{\left( 1-t\right) }\,.
\end{equation*}%
We first recall that $t=\frac{s_{0}-n}{s_{0}-2}$ and $1-t=\frac{n-2}{s_{0}-2}
$. Then, we use Young's inequality with exponents $\frac{s_{0}}{s_{0}-n}$
and $\frac{s_{0}}{n}$ 
\begin{equation}
\left( \int_{\Omega }\left( \eta ^{2}g^{\prime }\left( \Delta _{h}u\right)
\right) ^{\frac{s_{0}}{s_{0}-2}}\,dx\right) ^{\frac{s_{0}-2}{s_{0}}}
\label{new part 1}
\end{equation}%
\begin{equation*}
\leq \left( \lambda ^{-\frac{n}{s_{0}-n}}\int_{\Omega }\eta ^{2}g^{\prime
}\left( \Delta _{h}u\right) \,dx\right) ^{t\frac{s_{0}-2}{s_{0}}}\left(
\lambda ^{\frac{n}{n-2}}\int_{\Omega }\left( \eta ^{2}g^{\prime }\left(
\Delta _{h}u\right) \right) ^{\frac{n}{n-2}}\,dx\right) ^{\left( 1-t\right) 
\frac{s_{0}-2}{s_{0}}}
\end{equation*}%
\begin{equation*}
=\left( \lambda ^{-\frac{n}{s_{0}-n}}\int_{\Omega }\eta ^{2}g^{\prime
}\left( \Delta _{h}u\right) \,dx\right) ^{\frac{s_{0}-n}{s_{0}}}\left(
\lambda ^{\frac{n}{n-2}}\int_{\Omega }\left( \eta ^{2}g^{\prime }\left(
\Delta _{h}u\right) \right) ^{\frac{n}{n-2}}\,dx\right) ^{\frac{n-2}{s_{0}}}
\end{equation*}%
\begin{equation*}
\leq \tfrac{s_{0}-n}{s_{0}}\lambda ^{-\frac{n}{s_{0}-n}}\int_{\Omega }\eta
^{2}g^{\prime }\left( \Delta _{h}u\right) \,dx+\tfrac{n}{s_{0}}\lambda
\left( \int_{\Omega }\left( \eta ^{2}g^{\prime }\left( \Delta _{h}u\right)
\right) ^{\frac{n}{n-2}}\,dx\right) ^{\frac{n-2}{n}}\,.
\end{equation*}%
The function $\eta $ has compact support in $\Omega $ and we can apply the
Sobolev inequality with exponents $2$ and $2^{\ast }:=\frac{2n}{n-2}$ and
with the Sobolev constant $c=c(n)$ depending only on the dimension $n$ 
\begin{equation}
\left( \int_{\Omega }\left( \eta ^{2}g^{\prime }\left( \Delta _{h}u\right)
\right) ^{\frac{n}{n-2}}\,dx\right) ^{\frac{n-2}{n}}=\left( \int_{\Omega
}\left( \eta \left( g^{\prime }\left( \Delta _{h}u\right) \right) ^{\frac{1}{%
2}}\right) ^{2^{\ast }}\,dx\right) ^{\frac{2}{2^{\ast }}}  \label{new part 2}
\end{equation}%
\begin{equation*}
\leq c(n)\int_{\Omega }\left\vert D\left( \eta \left( g^{\prime }\left(
\Delta _{h}u\right) \right) ^{\frac{1}{2}}\right) \right\vert ^{2}\,dx\,.
\end{equation*}%
Collecting (\ref{new part 1}),(\ref{new part 2}) we get 
\begin{equation*}
\left( \int_{\Omega }\left( \eta ^{2}g^{\prime }\left( \Delta _{h}u\right)
\right) ^{\frac{s_{0}}{s_{0}-2}}\,dx\right) ^{\frac{s_{0}-2}{s_{0}}}\leq 
\tfrac{s_{0}-n}{s_{0}}\lambda ^{-\frac{n}{s_{0}-n}}\int_{\Omega }\eta
^{2}g^{\prime }\left( \Delta _{h}u\right) \,dx
\end{equation*}%
\begin{equation*}
+\tfrac{n}{s_{0}}\lambda c(n)\int_{\Omega }\left\vert D\left( \eta \left(
g^{\prime }\left( \Delta _{h}u\right) \right) ^{\frac{1}{2}}\right)
\right\vert ^{2}\,dx\,.
\end{equation*}%
A simple computation gives 
\begin{equation*}
D\left( \eta \left( g^{\prime }\left( \Delta _{h}u\right) \right) ^{\frac{1}{%
2}}\right) =D\eta \left( g^{\prime }\left( \Delta _{h}u\right) \right) ^{%
\frac{1}{2}}+\tfrac{1}{2}\eta \left( g^{\prime }\left( \Delta _{h}u\right)
\right) ^{-\frac{1}{2}}g^{\prime \prime }\left( \Delta _{h}u\right) \Delta
_{h}Du
\end{equation*}%
and we continue the previous estimate with 
\begin{equation*}
\left( \int_{\Omega }\left( \eta ^{2}g^{\prime }\left( \Delta _{h}u\right)
\right) ^{\frac{s_{0}}{s_{0}-2}}\,dx\right) ^{\frac{s_{0}-2}{s_{0}}}\leq
\int_{\Omega }\left( \tfrac{s_{0}-n}{s_{0}}\lambda ^{-\frac{n}{s_{0}-n}}\eta
^{2}+2\tfrac{n}{s_{0}}\lambda c(n)\left\vert D\eta \right\vert ^{2}\right)
g^{\prime }\left( \Delta _{h}u\right) \,dx
\end{equation*}%
\begin{equation}
+\tfrac{1}{2}\tfrac{n}{s_{0}}\lambda c(n)\int_{\Omega }\eta ^{2}\frac{\left(
g^{\prime \prime }\left( \Delta _{h}u\right) \right) ^{2}}{g^{\prime }\left(
\Delta _{h}u\right) }\left\vert \Delta _{h}Du\right\vert ^{2}\,dx\,.
\label{new part 3}
\end{equation}%
We recall that, for generic $k\in \mathbb{N}$ and $\beta \geq 0$, $g\left(
t\right) =t\left( 1+t^{2}\right) ^{\beta /2}$ when $t\in \left[ -k,k\right] $
and $g\left( t\right) $ affine out of the interval $\left[ -k,k\right] $ in
such a way that globally $g\in C^{1}\left( \mathbb{R}\right) $. Its
derivative $g^{\prime }\left( t\right) $ is continuos in $\mathbb{R}$, and
in fact it is Lipschitz continuos in $\mathbb{R}$. For $t\in \left[ -k,k%
\right] $ we have 
\begin{equation*}
\left\{ 
\begin{array}{l}
g^{\prime }\left( t\right) =\left( 1+t^{2}\right) ^{\beta /2-1}\left(
1+\left( \beta +1\right) t^{2}\right) \\ 
g^{\prime \prime }\left( t\right) =\beta t\left( 1+t^{2}\right) ^{\beta
/2-2}\left( 3+\left( \beta +1\right) t^{2}\right)%
\end{array}%
\right. ;
\end{equation*}%
while when $t\notin \left[ -k,k\right] $ then $g^{\prime }\left( t\right) $
is constant and $g^{\prime \prime }\left( t\right) =0$. We obtain 
\begin{equation*}
\tfrac{\left\vert g^{\prime \prime }\left( t\right) \right\vert }{g^{\prime
}\left( t\right) }\leq \tfrac{\beta \left\vert t\right\vert \left(
1+t^{2}\right) ^{\beta /2-2}\left( 3+\left( \beta +1\right) t^{2}\right) }{%
\left( 1+t^{2}\right) \left( 1+t^{2}\right) ^{\beta /2-2}\left( 1+\left(
\beta +1\right) t^{2}\right) }\leq 3\beta \frac{\left\vert t\right\vert }{%
1+t^{2}}\leq \tfrac{3}{2}\beta
\end{equation*}%
and then, taking the square of both sides, $\tfrac{\left( g^{\prime \prime
}\left( \Delta _{h}u\right) \right) ^{2}}{g^{\prime }\left( \Delta
_{h}u\right) }\leq \tfrac{9}{4}\beta ^{2}g^{\prime }\left( \Delta
_{h}u\right) $. Going back to (\ref{new part 3}) we obtain 
\begin{equation*}
\left( \int_{\Omega }\left( \eta ^{2}g^{\prime }\left( \Delta _{h}u\right)
\right) ^{\frac{s_{0}}{s_{0}-2}}\,dx\right) ^{\frac{s_{0}-2}{s_{0}}}\leq
\int_{\Omega }\left( \tfrac{s_{0}-n}{s_{0}}\lambda ^{-\frac{n}{s_{0}-n}}\eta
^{2}+2\tfrac{n}{s_{0}}\lambda c(n)\left\vert D\eta \right\vert ^{2}\right)
g^{\prime }\left( \Delta _{h}u\right) \,dx
\end{equation*}%
\begin{equation}
+\tfrac{9}{8}\tfrac{n}{s_{0}}\lambda \beta ^{2}c(n)\int_{\Omega }\eta
^{2}g^{\prime }\left( \Delta _{h}u\right) \left\vert \Delta
_{h}Du\right\vert ^{2}\,dx\,.  \label{Section 6 - formula 9}
\end{equation}%
Here we change notation, in principle by posing $\mu =\lambda \beta ^{2}$
for every $\beta >0$ ; for $\beta =0$ there is not necessity of this change.
More precisely, with the aim to avoid the denominator $\beta ^{2}$, which is
not uniformly far from zero, we pose $\mu =\lambda \left( \beta
^{2}+1\right) $ and we increase the last addendum in (\ref{Section 6 -
formula 9}) by changing $\beta ^{2}$ with $\beta ^{2}+1$; then of course $%
\lambda =\frac{\mu }{\beta ^{2}+1}$ and $\lambda \leq \mu $ for all $\beta
\geq 0$. Thus 
\begin{equation*}
\left( \int_{\Omega }\left( \eta ^{2}g^{\prime }\left( \Delta _{h}u\right)
\right) ^{\frac{s_{0}}{s_{0}-2}}\,dx\right) ^{\frac{s_{0}-2}{s_{0}}}\leq
\int_{\Omega }\left( \tfrac{s_{0}-n}{s_{0}}\left( \tfrac{\beta ^{2}+1}{\mu }%
\right) ^{\frac{n}{s_{0}-n}}\eta ^{2}+2\tfrac{n}{s_{0}}\mu c(n)\left\vert
D\eta \right\vert ^{2}\right) g^{\prime }\left( \Delta _{h}u\right) \,dx
\end{equation*}%
\begin{equation}
+\tfrac{9}{8}\tfrac{n}{s_{0}}\mu c(n)\int_{\Omega }\eta ^{2}g^{\prime
}\left( \Delta _{h}u\right) \left\vert \Delta _{h}Du\right\vert ^{2}\,dx\,.
\label{new part 4}
\end{equation}%
By (\ref{Section 6 - formula 5}) and (\ref{new part 4}) we get the final
estimate 
\begin{equation}
\left\vert \int_{\Omega }\left\vert b_{0}\left( x\right) \right\vert \cdot
\left\vert \Delta _{-h}\left( \eta ^{2}g\left( \Delta _{h}u\right) \right)
\right\vert \,dx\right\vert  \label{Section 6 - formula 10}
\end{equation}

\begin{equation}
\leq \int_{0}^{1}\,dt\int_{\Omega }\left\vert D\eta \right\vert ^{2}\frac{%
g^{2}\left( \Delta _{h}u\right) }{g^{\prime }\left( \Delta _{h}u\right) }%
\,dx+\varepsilon \int_{\Omega }\eta ^{2}g^{\prime }\left( \Delta
_{h}u\right) \left\vert \Delta _{h}u_{x_{s}}\right\vert ^{2}\,dx
\label{Section 6 - formula 11}
\end{equation}%
\begin{equation}
+\left( 1+\tfrac{1}{4\varepsilon }\right) \left\Vert b_{0}\right\Vert
_{L^{s_{0}}\left( \Omega ^{\prime }\right) }^{2}\left\{ c\left(
n,s_{0},\beta ,\mu \right) \int_{0}^{1}\,dt\int_{\Omega }(\eta
^{2}+\left\vert D\eta \right\vert ^{2})g^{\prime }\left( \Delta _{h}u\right)
\,dx\right.  \label{Section 6 - formula 12}
\end{equation}%
\begin{equation}
\left. +\tfrac{9}{8}\tfrac{n}{s_{0}}\mu c(n)\int_{0}^{1}\,dt\int_{\Omega
}\eta ^{2}g^{\prime }\left( \Delta _{h}u\right) \left\vert \Delta
_{h}Du\right\vert ^{2}\,dx\right\} \,,  \label{Section 6 - formula 13}
\end{equation}%
where $c\left( n,s_{0},\beta ,\mu \right) =\max \left\{ \tfrac{s_{0}-n}{s_{0}%
}\left( \tfrac{\beta ^{2}+1}{\mu }\right) ^{\frac{n}{s_{0}-n}};\tfrac{2n}{%
s_{0}}\mu c(n)\right\} $ depends only on $n,s_{0},\beta ,\mu $; in
particular it depends on powers of the parameter $\beta .$ We observe that
this constant diverge to $+\infty $ as $\mu \rightarrow 0^{+}$, but this
fact will not be a problem, since in the next section we will fix a
(sufficiently small) value of $\mu >0$.

\subsection{Conclusion}

As explained below, all the addenda in (\ref{Section 6 - formula 11}),(\ref%
{Section 6 - formula 12}),(\ref{Section 6 - formula 13}) can be reabsorbed
in (\ref{stima-Paolo}) and we obtain (cfr. with Section 5.4 in \cite%
{Marcellini 2023}) 
\begin{equation*}
\frac{1}{c}\int_{0}^{1}dt\int_{\Omega }\eta ^{2}g^{\prime }\left( \Delta
_{h}u\right) \left( 1+\left\vert \left( 1-t\right) Du\left( x\right)
+tDu\left( x+the_{s}\right) \right\vert ^{2}\right) ^{\frac{p-2}{2}%
}\left\vert \Delta _{h}Du\right\vert ^{2}\,dx\,
\end{equation*}%
\begin{equation}
\leq \int_{0}^{1}dt\int_{\Omega }\left( \eta ^{2}+\left\vert D\eta
\right\vert ^{2}\right) g^{\prime }\left( \Delta _{h}u\right) \left(
1+\left\vert Du\left( x\right) \right\vert ^{2}+\left\vert Du\left(
x+the_{s}\right) \right\vert ^{2}\right) ^{\frac{q}{2}}\,dx\,
\label{second derivatives}
\end{equation}%
\begin{equation*}
+\int_{0}^{1}dt\int_{\Omega }2\eta \left\vert D\eta \right\vert \cdot
\left\vert g\left( \Delta _{h}u\right) \right\vert \left( 1+\left\vert
Du\left( x\right) \right\vert ^{2}+\left\vert Du\left( x+the_{s}\right)
\right\vert ^{2}\right) ^{\frac{q-1}{2}}\,dx\,
\end{equation*}%
\begin{equation*}
+\int_{0}^{1}dt\int_{\Omega }\left\vert D\eta \right\vert ^{2}\cdot \frac{%
g^{2}\left( \Delta _{h}u\right) }{g^{\prime }\left( \Delta _{h}u\right) }%
\left( 1+\left\vert Du\left( x\right) \right\vert ^{2}+\left\vert Du\left(
x+the_{s}\right) \right\vert ^{2}\right) ^{\frac{q-2}{2}}\,dx\,.
\end{equation*}%
In particular the $\varepsilon -$addendum in (\ref{Section 6 - formula 11})
can be reabsorbed in the left side of (\ref{second derivatives}) if $%
\varepsilon $ is sufficiently small; then the addendum in (\ref{Section 6 -
formula 13}) above, although with the large factor $\left( 1+\frac{1}{%
4\varepsilon }\right) $ in front (however now with $\varepsilon $ fixed),
can be reabsorbed in the left side of (\ref{second derivatives}) by
considering $\mu $ sufficiently small. The dependence of the right hand side
of the estimate (\ref{second derivatives}) on powers of $\beta $ (precisely,
the dependence of the constant $c$ on powers of $\beta $) is allowed.

It remains only to follow the argument of \cite{Marcellini 2023} (see also
details in \cite[Section 4]{Marcellini JMAA 2021}) to conclude that the
gradient $Du$ of the weak solution is locally bounded in $\Omega $, as in (%
\ref{interpolation bound 2}). In fact, by Theorem 3.3 in \cite{Marcellini
2023} we can say that there exist constants $c,\alpha ,\gamma ,R_{0}>0$\
(depending on the $L^{\infty }\left( \Omega ^{\prime }\right) $ norm of $u$
and on the data, but not on $u$) such that, for every $\varrho $\ and $R$\
such that $0<\rho <R<R_{0}$, 
\begin{equation*}
\left\Vert Du\left( x\right) \right\Vert _{L^{\infty }\left( B_{\varrho };%
\mathbb{R}^{n}\right) }\leq \left( \tfrac{c}{\left( R-\varrho \right) ^{%
\frac{\gamma q}{\vartheta p}}}\left\Vert \left( 1+\left\vert Du\left(
x\right) \right\vert ^{2}\right) ^{\frac{1}{2}}\right\Vert _{L^{p}\left(
B_{R}\right) }\right) ^{\alpha }
\end{equation*}%
\begin{equation}
\underset{\text{for }n>2}{=}\;\left( \tfrac{c}{\left( R-\varrho \right) ^{%
\frac{\gamma q}{\vartheta p}}}\left\Vert \left( 1+\left\vert Du\left(
x\right) \right\vert ^{2}\right) ^{\frac{1}{2}}\right\Vert _{L^{p}\left(
B_{R}\right) }\right) ^{\frac{2p}{\left( n+2\right) p-nq}};
\label{interpolation bound 2 in the proof}
\end{equation}%
The explicit expression of the exponent $\alpha $ above (\ref{interpolation
bound 2 in the proof}) is given in (\ref{alpha}), with $\vartheta :=\frac{%
2^{\ast }-2}{2^{\ast }\frac{p}{q}-2}\;\;\underset{\text{for }n>2}{=}\;\;%
\frac{2q}{np-\left( n-2\right) q}$ and $\gamma =\frac{n}{q}\vartheta $.
Therefore $\frac{\gamma q}{\vartheta p}=\frac{n}{p}$ and also, from (\ref%
{interpolation bound 2 in the proof}), 
\begin{equation}
\left\Vert Du\left( x\right) \right\Vert _{L^{\infty }\left( B_{\varrho };%
\mathbb{R}^{n}\right) }\leq \left( \tfrac{c}{\left( R-\varrho \right) ^{n}}%
\int_{B_{R}}\left( 1+\left\vert Du\left( x\right) \right\vert ^{2}\right) ^{%
\frac{p}{2}}\,dx\right) ^{\frac{\alpha }{p}},
\label{interpolation bound 2 bis}
\end{equation}%
which correspond to the stated estimate (\ref{interpolation bound 2}).

The $W_{\mathrm{loc}}^{2,2}\left( \Omega \right) -$bound stated in (\ref%
{bound on second derivatives}) can be similarly obtained in this way: we
first use the bound (\ref{second derivatives}) with $g\left( t\right) $ as
above: $g\left( t\right) =t\left( 1+t^{2}\right) ^{\beta /2}$, $\beta \geq 0$%
.

In the special case $\beta =0$ we have $g\left( t\right) =t$, $g^{\prime
}\left( t\right) =1$ and 
\begin{equation*}
\frac{1}{c}\int_{0}^{1}dt\int_{\Omega }\eta ^{2}\left( 1+\left\vert \left(
1-t\right) Du\left( x\right) +tDu\left( x+the_{s}\right) \right\vert
^{2}\right) ^{\frac{p-2}{2}}\left\vert \Delta _{h}Du\right\vert ^{2}\,dx\,
\end{equation*}%
\begin{equation*}
\leq \int_{0}^{1}dt\int_{\Omega }\left( \eta ^{2}+\left\vert D\eta
\right\vert ^{2}\right) \left( 1+\left\vert Du\left( x\right) \right\vert
^{2}+\left\vert Du\left( x+the_{s}\right) \right\vert ^{2}\right) ^{\frac{q}{%
2}}\,dx\,
\end{equation*}%
\begin{equation*}
+\int_{0}^{1}dt\int_{\Omega }2\eta \left\vert D\eta \right\vert \cdot
\left\vert \Delta _{h}u\right\vert \left( 1+\left\vert Du\left( x\right)
\right\vert ^{2}+\left\vert Du\left( x+the_{s}\right) \right\vert
^{2}\right) ^{\frac{q-1}{2}}\,dx\,
\end{equation*}%
\begin{equation*}
+\int_{0}^{1}dt\int_{\Omega }\left\vert D\eta \right\vert ^{2}\cdot \left(
\Delta _{h}u\right) ^{2}\left( 1+\left\vert Du\left( x\right) \right\vert
^{2}+\left\vert Du\left( x+the_{s}\right) \right\vert ^{2}\right) ^{\frac{q-2%
}{2}}\,dx\,.
\end{equation*}%
Similarly to \cite{Marcellini 2023} we can go to the limit as $h\rightarrow
0 $. In the left hand side we go to the limit by lower semicontinuity and in
the limit $\left\vert D^{2}u\right\vert ^{2}$ appears. In the limit as $%
h\rightarrow 0$ all the three integrands in the right hand side can be
estimated by the $q-$power of the gradient $Du$. More precisely, in the
limit as $h\rightarrow 0$ we obtain (cfr. with (5.18) in \cite[Remark 5.1]%
{Marcellini 2023}) 
\begin{equation*}
\frac{1}{c}\int_{0}^{1}dt\int_{\Omega }\eta ^{2}\left( 1+\left\vert Du\left(
x\right) \right\vert ^{2}\right) ^{\frac{p-2}{2}}\left\vert
D^{2}u\right\vert ^{2}\,dx
\end{equation*}%
\begin{equation}
\leq \int_{0}^{1}dt\int_{\Omega }\left( \eta ^{2}+\left\vert D\eta
\right\vert ^{2}\right) \left( 1+\left\vert Du\left( x\right) \right\vert
^{2}\right) ^{\frac{q}{2}}\,dx.  \label{g(t)=t simplified}
\end{equation}%
The integral with respect to $t\in \left[ 0,1\right] $ is not more
necessary. We fix $\eta $ as usual. Precisely we consider concentric balls $%
B_{R}$ and $B_{\rho }$ compactly contained in $\Omega $, with $\rho
<R<R_{0}=R_{0}\left( \varepsilon ,n,s_{0}\right) $; then we consider a test
function $\eta \in C_{0}^{1}\left( B_{R}\right) $, $0\leq \eta \leq 1$ in $%
B_{R}$, $\eta =1$ in $B_{\rho }$ and $\left\vert D\eta \right\vert \leq
2/\left( R-\rho \right) $. We obtain the simplified version of (\ref{g(t)=t
simplified}) 
\begin{equation}
\int_{B_{\rho }}\left\vert D^{2}u\right\vert ^{2}\,dx\leq c\left( 1+\tfrac{4%
}{\left( R-\rho \right) ^{2}}\right) \int_{B_{R}}\left( 1+\left\vert
Du\left( x\right) \right\vert ^{2}\right) ^{\frac{q}{2}}\,dx\,.
\label{interpolation bound 0}
\end{equation}%
Since $0<\rho <R<R_{0}$, then $\tfrac{4}{\left( R-\rho \right) ^{2}}\geq 
\tfrac{4}{R_{0}^{2}}$ and thus $1\leq \tfrac{R_{0}^{2}}{4}\tfrac{4}{\left(
R-\rho \right) ^{2}}$. Therefore, with a different constant which we
continue to denote by $c$, we also have 
\begin{equation}
\int_{B_{\rho }}\left\vert D^{2}u\right\vert ^{2}\,dx\leq \tfrac{c}{\left(
R-\rho \right) ^{2}}\int_{B_{R}}\left( 1+\left\vert Du\left( x\right)
\right\vert ^{2}\right) ^{\frac{q}{2}}\,dx\,.  \label{interpolation 1}
\end{equation}%
Therefore the first $W_{\mathrm{loc}}^{2,2}\left( \Omega \right) -$bound
stated in (\ref{bound on second derivatives}) is obtained. We now make use
of the interpolation formula in \cite[Remark 6.1]{Marcellini 2023} (see also 
\cite[Theorem 3.1, formula (3.4)]{Marcellini 1991}) 
\begin{equation}
\left\Vert \left( 1+\left\vert Du\left( x\right) \right\vert ^{2}\right) ^{%
\frac{1}{2}}\right\Vert _{L^{q}\left( B_{\varrho }\right) }\leq \left( 
\tfrac{c}{\left( R-\varrho \right) ^{\gamma \left( \frac{q}{p}-1\right) }}%
\left\Vert \left( 1+\left\vert Du\left( x\right) \right\vert ^{2}\right) ^{%
\frac{1}{2}}\right\Vert _{L^{p}\left( B_{R}\right) }^{\frac{1}{\vartheta }%
}\right) ^{\alpha },  \label{interpolation bound 1}
\end{equation}%
where $\alpha ,\vartheta $ are expressed in (\ref{alpha}) and $\gamma =\frac{%
n}{q}\vartheta $. We iterate (\ref{interpolation 1}),(\ref{interpolation
bound 1}) in $B_{\rho }$, $B_{\left( R+\rho \right) /2}$, $B_{R}$; more
precisely we consider (\ref{interpolation 1}) with the balls $B_{\rho }$, $%
B_{\left( R+\rho \right) /2}$ and (\ref{interpolation bound 1}) with $%
B_{\left( R+\rho \right) /2}$, $B_{R}$. With different constants $c$ we
obtain 
\begin{equation*}
\int_{B_{\rho }}\left\vert D^{2}u\right\vert ^{2}\,dx\leq \tfrac{c}{\left(
R-\rho \right) ^{2}}\left\Vert \left( 1+\left\vert Du\left( x\right)
\right\vert ^{2}\right) ^{\frac{1}{2}}\right\Vert _{L^{q}\left( B_{\left(
R+\rho \right) /2}\right) }^{q}
\end{equation*}%
\begin{equation*}
\leq \frac{c}{\left( R-\rho \right) ^{2}}\left( \tfrac{1}{\left( R-\varrho
\right) ^{\gamma \left( \frac{q}{p}-1\right) }}\left\Vert \left(
1+\left\vert Du\left( x\right) \right\vert ^{2}\right) ^{\frac{1}{2}%
}\right\Vert _{L^{p}\left( B_{R}\right) }^{\frac{1}{\vartheta }}\right)
^{\alpha q}\,
\end{equation*}%
\begin{equation*}
=\frac{c}{\left( R-\rho \right) ^{2}}\left( \tfrac{1}{\left( R-\varrho
\right) ^{\gamma \vartheta \left( q-p\right) }}\int_{B_{R}}\left(
1+\left\vert Du\left( x\right) \right\vert ^{2}\right) ^{\frac{p}{2}%
}\,dx\right) ^{\frac{\alpha q}{\vartheta p}}
\end{equation*}%
\begin{equation}
=\tfrac{c}{\left( R-\rho \right) ^{2+\alpha \gamma q\left( \frac{q}{p}%
-1\right) }}\left\Vert \left( 1+\left\vert Du\left( x\right) \right\vert
^{2}\right) ^{\frac{1}{2}}\right\Vert _{L^{p}\left( B_{R}\right) }^{\frac{%
\alpha q}{\vartheta }}\,.\,  \label{final bound for the second derivatives}
\end{equation}%
This is the conclusion of the $W_{\mathrm{loc}}^{2,2}\left( \Omega \right) -$%
estimate, as stated in (\ref{bound on second derivatives}). Note that in the
special case $q=p$ all the parameters in this estimate simplify, $\alpha
=\vartheta =1$, and the bounds (\ref{bound on second derivatives for q=p}),(%
\ref{interpolation 1}) are reproduced.

The proof of Theorem \ref{final theorem} is now complete.

The computations in this Section \ref{Section: Local Lipschitz continuity}
are useful for proving Theorem \ref{final theorem 2} too. More precisely we
are now under the ellipticity (\ref{ellipticity}) and the $p,q-$growth
conditions (\ref{growth assumptions 2}),(\ref{growth assumptions 3}). In
particular, other than the ellipticity assumption (\ref{ellipticity}), with
the growth conditions%
\begin{equation}
\left\{ 
\begin{array}{l}
\left\vert \frac{\partial a^{i}}{\partial \xi _{j}}\right\vert \leq
M(1+\left\vert \xi \right\vert ^{2})^{\frac{q-2}{2}}+M\left\vert
u\right\vert ^{\alpha } \\ 
\left\vert \frac{\partial a^{i}}{\partial u}\right\vert \leq M(1+\left\vert
\xi \right\vert ^{2})^{\frac{q-2}{2}}+M\left\vert u\right\vert ^{\beta -1}
\\ 
\left\vert \frac{\partial a^{i}}{\partial x_{s}}\right\vert _{\;}\leq
M\left( L\right) (1+\left\vert \xi \right\vert ^{2})^{\frac{q-1}{2}} \\ 
\left\vert a^{i}\left( x,0,0\right) \right\vert \in L_{\mathrm{loc}}^{\gamma
^{\;}} \\ 
\left\vert b\left( x,u,\xi \right) \right\vert \leq M(1+\left\vert \xi
\right\vert ^{2})^{\frac{q-1}{2}}+M\left\vert u\right\vert ^{\delta
-1}+b_{0}\left( x\right) \,%
\end{array}%
\right.  \label{growth assumptions 2 in the proof}
\end{equation}%
Conditions (\ref{growth assumptions 2 in the proof})$_{2}$,(\ref{growth
assumptions 2 in the proof})$_{3}$,(\ref{growth assumptions 2 in the proof})$%
_{5}$ respectively correspond to (\ref{growth 0-b}),(\ref{growth 2}),(\ref{b
locally bounded}) when we replace $\frac{p+q}{2}$ by $q$. For a better
understanding, let us denote by $r:=\frac{p+q}{2}$; then in accordance $%
q=2r-p$. This means that, if for instance (\ref{growth assumptions 2 in the
proof})$_{2}$ corresponds to (\ref{growth 0-b}) when we replace $\frac{p+q}{2%
}$ by $q$, likewise in (\ref{growth assumptions 2 in the proof})$_{1}$ we
should replace $q$ with $2q-p$. In fact, since $2q-p\geq q$, then if (\ref%
{growth assumptions 2 in the proof})$_{1}$ holds, then all the more so 
\begin{equation*}
\left\vert \frac{\partial a^{i}}{\partial \xi _{j}}\right\vert \leq
M(1+\left\vert \xi \right\vert ^{2})^{\frac{2q-p-2}{2}}+M\left\vert
u\right\vert ^{\alpha }
\end{equation*}%
when we limit $\alpha $ (recall that, in Theorem \ref{final theorem}, $0\leq
\alpha \leq \frac{2\left( q-2\right) }{q-p+2}$) with the corresponding bound 
\begin{equation*}
0\leq \alpha \leq \tfrac{2\left\{ \left( 2q-p\right) -2\right\} }{\left(
2q-p\right) -p+2}=\tfrac{2q-p-2}{q-p+1}\,,
\end{equation*}%
as stated in (\ref{growth assumptions 3})$_{1}$. Note that, when $q=p$, the
two constraints for $\alpha $ coincide each other. About the condition $%
\frac{q}{p}<1+\frac{1}{n}$, when we replace $q$ by $2q-p$ we obtain $\frac{%
2q-p}{p}<1+\frac{1}{n}$; i.e., $\frac{q}{p}<1+\frac{1}{2n}$. Finally the
exponent in the $W^{1,\infty }-$local estimate (\ref{interpolation bound 3}%
), which in Theorem \ref{final theorem} when $n\geq 3$ was equal to $\frac{2p%
}{\left( n+2\right) p-nq}$, now becomes $\frac{p}{\left( n+1\right) p-nq}$.
The proof of Theorem \ref{final theorem 2} is complete too.

\bigskip

\bigskip

\textbf{Acknowledgement} \ The authors are members of the \textit{Gruppo
Nazionale per l'Analisi Matematica, la Probabilit\`{a} e le loro
Applicazioni (GNAMPA)} of the \textit{Istituto Nazionale di Alta Matematica
(INdAM)}.

\end{document}